\documentclass[11pt,a4paper]{amsart}
\usepackage{amssymb}
\usepackage{latexsym}
\usepackage{exscale}
\usepackage{amsfonts}
\usepackage{graphicx}
\usepackage{mathrsfs}
\usepackage{amsmath,amscd,amsthm}
\usepackage{bbm,pifont}
\usepackage{enumerate}
\usepackage{color}
\usepackage{amssymb}
\usepackage{latexsym}
\usepackage{exscale}
\usepackage[utf8]{inputenc}
\usepackage[T1]{fontenc}
\usepackage{dsfont}
\usepackage[mathscr]{eucal}

\allowdisplaybreaks

\DeclareMathAlphabet\mathcalbf{OMS}{cmsy}{b}{n}
\DeclareMathAlphabet\EuScript{U}{eus}{m}{n}
\DeclareMathAlphabet\EuScriptBold{U}{eus}{b}{n}

\numberwithin{equation}{section}
\newtheorem{theorem}{Theorem}[section]
\newtheorem{lemma}[theorem]{Lemma}

\newtheorem{proposition}[theorem]{Proposition}

\newtheorem{definition}[theorem]{Definition}

\newtheorem{remark}[theorem]{Remark}

\def\C{\mathbb C}

\begin{document}
\allowdisplaybreaks

\title[Commutators of Cauchy--Szeg\H{o} type integrals for domains in $\C^n$]
{Commutators of Cauchy--Szeg\H{o} type integrals for domains  in $\C^n$ with minimal smoothness}

\author{Xuan Thinh Duong, Michael Lacey, Ji Li, Brett D. Wick and Qingyan Wu}

\address{Xuan Thinh Duong, Department of Mathematics, Macquarie University, NSW, 2109, Australia}
\email{xuan.duong@mq.edu.au}

\address{Michael Lacey, School of Mathematics \\
         Georgia Institute of Technology\\
         Atlanta, GA 30332, USA
         }
\email{lacey@math.gatech.edu}
%\thanks{Research supported in part by grant  from the US National Science Foundation, DMS-1600693 and the 
%Australian Research Council ARC DP160100153.}

\address{Ji Li, Department of Mathematics, Macquarie University, NSW, 2109, Australia}
\email{ji.li@mq.edu.au}

\address{Brett D. Wick, Department of Mathematics \& Statistics\\
         Washington University - St. Louis\\
         St. Louis, MO 63130-4899 USA
         }
\email{wick@math.wustl.edu}

\address{Qingyan Wu, Department of Mathematics\\
         Linyi University\\
         Shandong, 276005, China
         }
\email{wuqingyan@lyu.edu.cn}

%    General info
%\date{\today}

\begin{abstract}
In this paper we study the commutator of Cauchy type integrals $\EuScript C$  on a bounded strongly pseudoconvex domain $D$ in $\C^n$ with boundary $bD$ satisfying the minimum regularity condition $C^{2}$ as in the recent result of Lanzani--Stein. We point out that in this setting the Cauchy type integrals $\EuScript C$ is the sum of the essential part $\EuScript C^\sharp$ which is a Calder\'on--Zygmund operator and a remainder $\EuScript R$ which is no longer a Calder\'on--Zygmund operator. We show that the commutator $[b, \EuScript C]$  is bounded on $L^p(bD)$ ($1<p<\infty$)  if {\color{black}and only if}\ $b$ is in the BMO space on $bD$.  Moreover, the commutator $[b, \EuScript C]$  is compact on $L^p(bD)$ ($1<p<\infty$)  if {\color{black}and only if}\ $b$ is in the VMO space on $bD$.  Our method can also be applied to the commutator of Cauchy--Leray integral  in a bounded, strongly $\C$-linearly convex domain $D$ in $\C^n$ with the boundary $bD$ satisfying the minimum regularity $C^{1,1}$. Such a Cauchy--Leray integral is a Calder\'on--Zygmund operator as proved in the recent result of Lanzani--Stein. We also point out that our method provides another proof of the boundedness and compactness of commutator of Cauchy--Szeg\H o operator on a bounded strongly pseudoconvex domain $D$ in $\C^n$ with smooth boundary (first established by Krantz--Li).

\end{abstract}

\maketitle
%---------------------------------------------------------1----------------------------------------

%\tableofcontents

\section{Introduction and statement of main results}
\setcounter{equation}{0}

The theory of Hardy spaces originated from the study of functions on the complex plane.
Denote the open unit disc in the complex plane by $\mathbb{D}=\{ z\in\mathbb{C}:\ |z|<1\}$.
We recall that the classical Hardy space $H^p$, $0<p<\infty$, on $\mathbb{D}$
 is defined as the space of holomorphic functions $f$ that satisfy $ \|f\|_{H^p(\mathbb{D})} < \infty$, where
$$  \|f\|_{H^p(\mathbb{D})}:=\sup_{0\leq r<1} \Big( {1\over 2\pi}\int_0^{2\pi}  \big|f(re^{it})\big|^p dt \Big)^{1\over p}.  $$
It is easy to check that the pointwise product of two
 $H^2(\mathbb{D})$ functions is a function in the Hardy space $H^1(\mathbb{D})$. The converse is not obvious but actually is  true and  we have  the important
 {\it   Riesz factorization theorem}:
 {\it ``A function $f$ is in $H^1(\mathbb{D})$ if and only if there exist $g,h\in H^2(\mathbb{D})$ with
 $f= g\cdot h$ and $\|f\|_{H^1(\mathbb{D})} =\|g\|_{H^2(\mathbb{D})}\|h\|_{H^2(\mathbb{D})}$.''
  }

A similar result holds for the  Hardy space $H^1$ on the unit circle $\mathbb {T} = \{ z\in\mathbb{C}:\ |z|= 1\}$. This factorisation has a key role in 
proving the equivalence of norms
 $$ \|[b,H]\|_{L^2(\mathbb{T})\to L^2(\mathbb{T})} \approx \|b\|_{\rm BMO(\mathbb T)}, $$
 where $[b,H](f)=b H(f) - H(bf)$ is the commutator of  a BMO function $b$ and the Hilbert transform $H$ on the unit circle. We note that this result can be interpreted
through Hankel operators, and one then recovers a famous result of Nehari \cite{N}.
See \cite{L} for the history and literature of the Nehari theorem. See also  \cite{CRW} for the norm equivalence $ \|[b,R_j]\|_{L^2(\mathbb{R}^n)\to L^2(\mathbb{R}^n)} \approx \|b\|_{{\rm BMO}(\mathbb R^n)}, $
where $R_j$ $(j=1,\ldots,n)$ is the $j$th Riesz transform on the Euclidean space $\mathbb R^n$, and \cite{CLMS} for the application of commutator to certain version of div-curl lemmas.

Related estimates on commutators have been studied extensively in different settings, see for example  \cite{B, DHLWY, DLLW,DLWY, GLW, HLW, Hy, IR, LOR0, LOR, LNWW,LW,P,TYY, U, DGKLWY}
and the references therein.

We now recall the extension of this fundamental commutator result to the setting of several complex variables.
Let $D$ be a bounded domain in $\C^n$ with $C^2$ boundary $bD$, $d\sigma$ the Lebesgue surface measure on $bD$ and $L^p(bD)$ the usual Lebesgue space on $bD$ with respect to the measure $d\sigma$. Let $H^p(D)$ be the holomorphic Hardy spaces defined in \cite{KRA2,STE1}. Fatou's theorem \cite{KRA1} shows that, for any $0 < p \leq\infty$, a holomorphic function $f\in H^p(D)$ has a radial limit at almost all points on $bD$. It then follows from the maximum principle that one can identify $H^p(D)$ as a closed subspace of $L^p(bD)$. Let $S: L^2(bD)\to H^2(D)$ be the orthogonal projection via the reproducing kernel $S(z,w)$ which is known as the Cauchy--Szeg\H o kernel. For a number of  special cases and classes of domains $D$, we may identify the operator $S$ as a singular integral operator on $bD$; in fact, in many instances the Cauchy--Szeg\H o kernel $S(z,w)$  is $C^\infty$ on $bD\times bD\backslash \{z=w\}$.

Recall that when $D$ is a strictly pseudoconvex domain in $\C^n $ with smooth boundary,
let $T_S$ be the singular integral associated with the Cauchy--Szeg\H o kernel, in fact one has $S(f)(z)={1\over 2}f(z)+cT_S(f)$ for almost every $z\in bD$. %(see \cite{KL2}). 
  Krantz and Li \cite{KL2} first proved the following result regarding the boundedness of the commutator of $T_S$ with respect to the BMO space on the boundary $bD$,  
 as well as the compactness of commutator with respect to the VMO space. Here the BMO and VMO spaces were studied in \cite{KL1}.  To be more precise, 

\smallskip
\noindent{\bf Theorem A\,}(\cite{KL2}){\bf .} {\it Let  $D$ be a bounded strictly pseudoconvex domain in $\C^n $ with smooth boundary and let $b\in L^1(bD)$. Then for $1<p<\infty$,

{\rm(i)} 
$b\in {\rm BMO}(bD)$ if and only if $[b,T_S]$ is bounded on $L^p(bD)$.

{\rm(ii)} $b\in {\rm VMO}(bD)$ if and only if $[b,T_S]$ is compact on $L^p(bD)$.
}

\smallskip
We note that in the study of boundedness of the commutator $[b,T_S]$, the regularity of the kernel $S(z,w)$ plays a key role. In the proof of Theorem A, it follows from the results in \cite{BMS, FEF, NRSW} that
the Cauchy--Szeg\H o kernel $S(z,w)\in C^\infty(bD\times bD \backslash \{(z,z):\ z\in bD\})$ and is a standard Calder\'on--Zygmund kernel, hence the boundedness and compactness of the commutator follows from the standard approach as in \cite{CRW} with suitable modifications. Conversely, when $[b, T_S]$ is bounded on $L^p(bD)$, to show $b\in {\rm BMO}(bD)$,
Krantz--Li used the fact that  the reciprocal of the Cauchy--Szeg\H o kernel, ${1\over S(z,w)}$ has
 a decomposition as a finite sum of holomorphic functions.
 Hence, by writing $1 =S(z,w) \times {1\over S(z,w)} $ and by the decomposition of ${1\over S(z,w)}$, they could link the BMO norm to the commutator with some additional algebra. By using this technique they also showed that  when $[b,T_S]$ is compact, $b\in {\rm VMO}(bD)$.

\smallskip
Recently, Lanzani and Stein \cite{LS}
 studied the Cauchy--Szeg\H o projection operator in a bounded domain $D$ in $\C^n$ which is strongly pseudoconvex and its boundary $bD$ satisfies the minimum regularity condition of class $C^2$. The measure that they used on the boundary $bD$ is the Leray--Levi measure $d\lambda$ 
(for the details we refer to Section 2 below). They obtained the $L^p(bD)$ boundedness  ($1<p<\infty$){\  of a family of Cauchy integrals $\{\EuScript C_\epsilon\}_\epsilon$, and since the role of the parameter $\epsilon$ is of no consequence here, when denoting a member in this family we will simply write $\EuScript C$}. Here the space $L^p(bD)$ is with respect to $d\lambda$. 
We point out that the kernel of these Cauchy integral operators
 do not satisfy the standard size or smoothness conditions
for Calder\'on--Zygmund operators. To obtain the $L^p(bD)$ boundedness, 
they decomposed the Cauchy transform $\EuScript C$  which is the restriction of such a Cauchy integral on $bD$ into the essential part $\EuScript C^\sharp$ and the remainder $\EuScript R$, i.e.,
$$   \EuScript C =  \EuScript C^\sharp+ \EuScript R,  $$
where the kernel of $\EuScript C^\sharp$, denoted by $C^\sharp(w,z)$, satisfies the standard size and smoothness conditions
for Calder\'on--Zygmund operators, i.e. there exists a positive constant $A_1$ such that for every $w,z\in bD$ with $w\not=z$,
\begin{align}\label{gwz}
\left\{
                \begin{array}{ll}
                  a)\ \ |C^\sharp(w,z)|\leq A_1 {\displaystyle1\over\displaystyle{\tt d}(w,z)^{2n}};\\[5pt]
                  b)\ \ |C^\sharp(w,z) - C^\sharp(w',z)|\leq A_1 {\displaystyle {\tt d}(w,w')\over \displaystyle {\tt d}(w,z)^{2n+1} },\quad {\rm if}\ {\tt d}(w,z)\geq c{\tt d}(w,w');\\[5pt]
                  c)\ \ |C^\sharp(w,z) - C^\sharp(w,z')|\leq A_1  {\displaystyle {\tt d}(z,z')\over \displaystyle {\tt d}(w,z)^{2n+1} },\quad {\rm if}\ {\tt d}(w,z)\geq c {\tt d}(z,z')
                \end{array}
              \right.
\end{align}
for an appropriate constant $c>0$ and where ${\tt d}(z,w)$ is a quasi-distance suitably adapted to $D$. And hence, the $L^p(bD)$ boundedness  ($1<p<\infty$) of $\EuScript C^\sharp$ follows from a version of $T(1)$ theorem.  However, the kernel $R(w,z)$ of $\EuScript R$ satisfies a size condition and a smoothness condition for only one of the variables as follows: there exists
a positive constant $C_R$ such that for every $w,z\in bD$ with $w\not=z$,
\begin{align}\label{cr}
\left\{
                \begin{array}{ll}
                  d)\ \ |R(w,z)|\leq C_R {\displaystyle1\over \displaystyle{\tt d}(w,z)^{2n-1}};\\[5pt]
                  e)\ \ |R(w,z)-R(w,z')|\leq C_R {\displaystyle{\tt d}(z,z')\over\displaystyle {\tt d}(w,z)^{2n}},\quad {\rm if\ }  {\tt d}(w,z)\geq c_R {\tt d}(z,z')
                \end{array}
              \right.
\end{align}
for an appropriate large constant $c_R$. It is worth pointing out that in the size condition and smoothness condition above, the dimensions are strictly smaller than the homogeneous dimension $2n$ of the boundary $bD$.  The $L^p(bD)$ boundedness  ($1<p<\infty$) of $\EuScript R$ follows from Schur's lemma. It is also worth to point out that the hypothesis of minimal smoothness is sharp, see more explanations and counterexamples  in \cite{LS2017} when the boundary $bD$ does not satisfy the $C^2$ smoothness.

Thus, along the literature of Nehari, {Coifman}--Rochberg--Weiss, Krantz--Li, it is natural to study the behavior of
the commutator of Cauchy type integrals as studied by Lanzani--Stein (\cite{LS}), which is not a standard Calder\'on--Zygmund operator,  with a BMO function for a bounded strongly pseudoconvex domain in $\C^n$ with minimal smoothness.

The main result of our paper is on the commutator of Cauchy transform  $\EuScript C$
(as in \cite{LS}). 
\begin{theorem}\label{cauchy}
Suppose $D\subset \mathbb C^n$, $n\geq 2$, is  a bounded domain whose boundary is of class $C^2$ and is strongly pseudoconvex.
Suppose $b\in L^1(bD, d\lambda)$. Then for $1<p<\infty$,

$(1)$  $b\in{\rm BMO}(bD,d\lambda)$ if and only if the commutator $[b, \EuScript C]$ is bounded on  $L^p(bD, d\lambda)$.

$(2)$   $b\in{\rm VMO}(bD,d\lambda)$ if and only if the commutator $[b, \EuScript C]$ is compact on  $L^p(bD, d\lambda)$. 
\end{theorem}

\color{black}

We first point out that the method in \cite{KL2} does not work in this setting, since in general there is no information about the reciprocal of the kernel of  Cauchy transform  $\EuScript C$.

To obtain the necessary condition in (1) above, we use the decomposition $\EuScript C = \EuScript C^\sharp + \EuScript R$ from \cite{LS}. Then when $b$ is in ${\rm BMO}(bD,d\lambda)$, we study the boundedness of commutators $[b,\EuScript C^\sharp]$ and $[b,\EuScript R]$, respectively.
For $[b,\EuScript C^\sharp]$,  the  upper bound follows directly from the known result in \cite[Theorerm 3.1]{KL1} since $\EuScript C^\sharp$ is a Calder\'on--Zygmund operator.
For $[b,\EuScript R]$, although the kernel of $\EuScript R$ does not satisfies the smoothness condition for the first variable and the dimension of the size condition does not match the homogeneous dimension,  we can still obtain the upper bound by using the condition that $D$ is a bounded domain and using the sharp maximal function and the John--Nirenberg inequality together with a suitable decomposition of the underlying space $bD$. 
Thus, combining the above two results,  we obtain the upper bound of the commutator of $[b,\EuScript C]$.

To prove the sufficient condition in (1) above, we first point out that comparing to the previous results such as in \cite{U,LOR,DLLW,TYY}, the kernel of the operator here does not satisfy the conditions such as  dilation invariance or sign invariance in a collection of chosen balls. Hence  we make good use of the explicit kernel condition of the essential part $\EuScript C^\sharp$
and the upper bound of the kernel of the remainder $\EuScript R$ as in $d)$ above, and then combine an idea from \cite{U} (see also \cite{LOR}) of using the median value for the definition of BMO space instead of average, and exploiting a suitable decomposition of the underlying domain to match the kernel condition.

To obtain the necessary condition in (2) above, again we point out that since $\EuScript C^\sharp$ is a Calder\'on--Zygmund operator, the proof follows from \cite{KL2}. It suffices to prove that $[b,\EuScript R]$ is compact when $b\in{\rm VMO}(bD,d\lambda)$. This follows from a standard approach via Ascoli--Arzela theorem, together with the specific conditions on the kernel of $\EuScript R$.

To prove the sufficient condition in (2) above, we note that the classical approach of \cite{U} does not apply since $\EuScript C$ is no longer a Calder\'on--Zygmund operator. To verify that $b\in{\rm VMO}(bD,d\lambda)$ when $[b,\EuScript R]$ is compact, the key steps are the following: i) our approach in the proof of the sufficient condition in (1) by using a suitable modification of decompositions; ii) a fundamental fact that there is no bounded operator $T \;:\; \ell^{p} (\mathbb N) \to \ell^{p} (\mathbb N)$ with $Te_{j } = T e_{k} \neq 0$ for all $j,k\in \mathbb N$.  Here, $e_{j}$ is the 
standard basis for $\ell^{p} (\mathbb N)$. It is worth to point out that
this proof here is new in the literature of compactness of commutators.

\begin{remark}\label{rem1}
In Theorem \ref{cauchy} we used the Leray--Levi measure $d\lambda$ for the $L^p$, BMO and VMO spaces. However, we point out that the result also holds for the $L^p$, BMO and VMO spaces
with respect to any measure $d\omega$ on $bD$ of the form $d\omega = \omega d\sigma$, where 
the density $\omega$ is a strictly positive continuous function on $bD$ $($as stated in \cite[Section 1.2]{LS}$)$. We provide the explanation after the proof of Theorem \ref{cauchy}.
\end{remark}

\begin{remark}\label{rem1 family}
Recall that the Cauchy integrals studied in \cite{LS} are a family of operators, and we only consider one member from this family.  However, in \cite{LS} they also pointed out that, if we work with a family of Cauchy integrals $\{\EuScript C_\epsilon\}_\epsilon$ parameterized by $0<\epsilon< \epsilon_0$, then by splitting $\EuScript C_\epsilon =  \EuScript C^\sharp_\epsilon+ \EuScript R_\epsilon$, the operator norm $\|\EuScript R_\epsilon\|_{L^p(bD,d\lambda)\to L^p(bD,d\lambda)} $ can go to $\infty$ as $\epsilon\to0$. Here we also point out that, this fact will affect the $L^p(bD,d\lambda)\to L^p(bD,d\lambda)$ norm of the commutator $[b, \EuScript C_\epsilon]$ as well. We provide the explanation in details at the end of Subsection 2.3.
\end{remark}

We also consider the Cauchy--Leray integral in the setting of Lanzani--Stein  \cite{LS2014}, where they studied the 
Cauchy--Leray integral in a bounded domain $D$ in $\C^n$ which is strongly $\C$-linearly convex and the boundary $bD$ satisfies the minimum regularity $C^{1,1}$  (for the details we refer to Section 3 below). They obtained the $L^p(bD)$ boundedness ($1<p<\infty$) of the  Cauchy--Leray transform $\mathcal C$ by showing that the kernel $K(w,z)$ of $\mathcal C$ satisfies the standard size and smoothness conditions of Calder\'on--Zygmund operators as in a), b) and c) above (for details of these definitions and notation, we refer the readers to Section 3), and 
that  $\mathcal C$ satisfies a  suitable version of $T(1)$ theorem.  Following a similar approach as in the proof for Theorem \ref{cauchy}, we arrive at the second main result of this paper on the commutator of the Cauchy--Leray transform (as in \cite{LS2014}).
\begin{theorem}\label{Cauchy-Leray}
Let $D$ be a bounded domain in $\mathbb C^n$ of class $C^{1,1}$ that is strongly $\mathbb C$-linearly convex and let $b\in L^1(bD, d\lambda)$. 
Let $\mathcal C$ be the Cauchy--Leray transform $($as in \cite{LS2014}$)$. Then for $1<p<\infty$, 

$(1)$  $b\in{\rm BMO}(bD,d\lambda)$ if and only if 
the commutator $[b, \mathcal C]$ is bounded on  $L^p(bD, d\lambda)$.  

$(2)$  $b\in{\rm VMO}(bD,d\lambda)$ if and only if 
the commutator $[b, \mathcal C]$ is compact on  $L^p(bD, d\lambda)$.  

\end{theorem}

{Similar to the remark for Theorem \ref{cauchy}, we have the following.
\begin{remark}\label{rem2}
In Theorem \ref{Cauchy-Leray} we used the Leray--Levi  measure $d\lambda$ for the $L^p$, BMO and VMO spaces. However, we point out that the result also holds for the $L^p$, BMO and VMO spaces
with respect to any measure $d\omega$ on $bD$ of the form $d\omega = \omega d\sigma$, where 
the density $\omega$ is a strictly positive continuous function on $bD$. We provide the explanation in details after the proof of Theorem \ref{Cauchy-Leray}.
\end{remark}}

This paper is organised as follows.  In Section 2 we recall the notation and definitions related to a family of Cauchy integrals for bounded strongly pseudoconvex domains in $\C^n$ with minimal smoothness,  
then we prove Theorem \ref{cauchy}. 
In Section 3 we recall the notation and definitions related to the Cauchy--Leray integral for bounded $\C$-linearly convex domains in $\C^n$ with minimal smoothness and we prove Theorem \ref{Cauchy-Leray}. In the last section we point out that our method here can provide a different proof of Theorem A for the Cauchu--Szeg\H o operator on a strictly pseudoconvex domain in $\C^n $ with smooth boundary, which was first obtained in \cite{KL2}
 
 Throughout this paper,  $c$ and $\tilde c$ will denote positive constants which are independent of the main parameters, but they may vary from line to line. For every $p\in (1,\infty)$, $p'$ means the conjugate of $p$, i.e., $1/p+1/p'=1$.
  By $f\lesssim g$,
we shall mean $f\le c g$ for some positive constant 
 $c$. If $f\lesssim g$ and $g\lesssim f$, we then write $f \approx g$.

\section{ Cauchy type integral for bounded strongly pseudoconvex domains in $\C^n$ with minimal smoothness}
\setcounter{equation}{0}

\subsection{Preliminaries}
The submanifolds we shall be interested in are the boundaries of appropriate domains $D\subset \mathbb C^n$.
More precisely, we consider a bounded domain $D$ with defining function $\rho$, {which means that 
$D=\{z\in\mathbb C^n: \rho(z)<0\}$ with $\rho: \mathbb C^n\rightarrow \mathbb R$ and $bD=\{w\in\mathbb C^n: \rho(w)=0\}$ with $\nabla \rho(w)\not=0$ for all $w\in bD$}. 

In this section, we 
always assume that $D$ is a bounded strongly pseudoconvex
domain {whose boundary is of class $C^2$, that is, $\rho\in C^2(\mathbb C^n,\mathbb R)$}.

We now recall the notation from  \cite{LS}. 
Since our domain is strongly pseudoconvex, we may assume without loss of generality that its defining function $\rho$ is strictly plurisubharmonic (see \cite{Ra}). The assumptions regarding the domain $D$ and $\rho$ will be in force throughout in this section and so will not be restated below.

Let $\mathcal L_0(w, z)$ be the negative of the Levi polynomial at $w\in bD$, given by
$$ \mathcal L_0(w, z) = \langle \partial\rho(w),w-z\rangle -{1\over2} \sum_{j,k} {\partial^2\rho(w) \over \partial w_j \partial w_k} (w_j-z_j)(w_k-z_k),  $$
where $\partial \rho(w)=({\partial\rho\over\partial w_1}(w),\cdots, {\partial\rho\over\partial w_n}(w))$
and we have used the notation $\langle\eta,\zeta\rangle=\sum_{j=1}^n\eta_j\zeta_j$ for $\eta=(\eta_1,\cdots, \eta_n), \zeta=(\zeta_1,\dots,\zeta_n)\in\mathbb C^n$.
The strict plurisubharmonicity of $\rho $ implies that
$$  2\operatorname{ Re} \mathcal L_0(w, z) \geq -\rho(z)+c |w-z|^2,  $$
for some $c>0$, whenever $w\in bD$ and $z\in \bar D$ is sufficiently close to $w$.
Then a modification of $\mathcal L_0$ is as follows
\begin{align}\label{g0}
    g_0(w,z) = \chi \mathcal L_0+ (1-\chi) |w-z|^2.  
\end{align}
Here $\chi=\chi(w,z)$ is a $C^\infty$-cutoff function with $\chi=1$ when $|w-z|\leq \mu/2$ and $\chi=0$ if $|w-z|\geq \mu$.
Then for $\mu$ chosen sufficiently small (and then kept fixed throughout), we have that
$$ \operatorname{ Re}g_0(w,z)\geq c(-\rho(z)+ |w-z|^2) $$
for $z$ in $\bar D$ and $w$ in $bD$, with $c$  a positive constant.

Note that the modified Levi polynomial $g_0$ has no smoothness beyond continuity in the variable $w$. So in \cite{LS}, for each $\epsilon>0$ they considered a variant $g_\epsilon$ defined as follows: let $\{\tau_{jk}^\epsilon(w)\}$ be an $n\times n$-matrix  of $C^1$ functions such that
$$\sup_{w\in bD}\Big|{\partial^2\rho(w)\over\partial w_j\partial w_k}- \tau_{jk}^\epsilon(w)\Big|\leq\epsilon,\quad 1\leq j,k\leq n.$$
Set
$$ \mathcal L_\epsilon(w, z) = \langle \partial\rho(w),w-z\rangle -{1\over2} \sum_{j,k}\tau_{jk}^\epsilon(w) (w_j-z_j)(w_k-z_k),  $$
and define
$$    g_\epsilon(w,z) = \chi \mathcal L_\epsilon+ (1-\chi) |w-z|^2, \quad z,w\in\mathbb C^n.  $$
Now $g_\epsilon$ is $C^1$ in $w$ and $C^\infty$ in $z$. We note that
$$\left|g_0(w,z)-g_\epsilon(w,z) \right|\lesssim \epsilon |w-z|^2.$$
We shall always assume that $\epsilon$ is sufficiently small, we then have 
$$\left|g_\epsilon(w,z)\right|\approx\left|g_0(w,z) \right|,$$
where the equivalence $\approx$ is independent of $\epsilon$.

Now on the boundary $bD$, define the function 
${\tt d}(w,z) = |g_0(w,z)|^{1\over2}$. According to \cite[Proposition 3]{LS},  ${\tt d}$ satisfies 
the following conditions: there exist constants $A_2>0$ and  $C_d>1$ such that for all $w,z,z'\in bD$,
\begin{align}\label{metric d}
\left\{
                \begin{array}{ll}
                  1)\ \ {\tt d}(w,z)=0\quad {\rm iff}\quad w=z;\\[5pt]
                  2)\ \ A_2^{-1} {\tt d}(z,w)\leq  {\tt d}(w,z) \leq A_2 {\tt d}(z,w);\\[5pt]
                  3)\ \ {\tt d}(w,z)\leq C_d\big( {\tt d}(w,z') +{\tt d}(z',z)\big).
                \end{array}
              \right.
\end{align}

\smallskip
Next we recall the Leray--Levi measure $d\lambda$ on $bD$ defined via the $(2n-1)$-form
$$ {1\over (2\pi i)^n} \partial \rho \wedge (\bar\partial \partial \rho)^{n-1}. $$
To be more precise, we have the linear functional
\begin{align}\label{lambda}  
f\mapsto {1\over (2\pi i)^n} \int_{bD} f(w) j^*(\partial \rho \wedge (\bar\partial \partial \rho)^{n-1})(w)=: \int_{bD} f(w)d\lambda(w) 
\end{align}
defined for $f\in C(bD)$, and this defines the measure $d\lambda$. 
Then one also has 
$$   d\lambda(w) = {1\over (2\pi i)^n} j^*(\partial\rho \wedge (\bar\partial \partial\rho)^{n-1}) (w) =\Lambda(w)d\sigma(w),
 $$
where $j^*$ denotes the pullback under the inclusion $j:bD\hookrightarrow\mathbb C^n$, $d\sigma$ is the induced Lebesgue measure on $bD$ and $\Lambda(w)$ is a continuous function such that
$ c\leq \Lambda(w)\leq \tilde c, w\in bD$, with $c$ and $\tilde c$ two positive constants.

We also recall the  boundary balls $ B_r(w) $ determined via the quasidistance ${\tt d }$ and their measures, i.e.,
\begin{align}\label{ball} 
B_r(w) :=\{ z\in bD:\ {\tt d}(z,w)<r \}, \quad {\rm where\ } w\in bD. 
\end{align} 
According to \cite[p. 139]{LS},
\begin{align}\label{lambdab}
c_\lambda^{-1} r^{2n}\leq \lambda\big(B_r(w) \big)\leq c_\lambda r^{2n},\quad 0<r\leq 1,
\end{align}
for some $c_\lambda>1$.

In \cite{LS} they defined the Cauchy integrals determined by the denominators $g_\epsilon$ and studied their properties when $\epsilon$ is kept fixed. For convenience of notation we will henceforth drop explicit reference to $\epsilon$. Thus, we will write $g$ for $g_\epsilon$ in the following.  To study the Cauchy  transform $\EuScript C$,  which is the restriction of such a Cauchy integral on $bD$,
one of the key steps in  \cite{LS} is that they provided a constructive decomposition of $ \EuScript C$ as follows:
$$   \EuScript C =  \EuScript C^\sharp+ \EuScript R,  $$
where the essential part 
\begin{align}
 \EuScript C^\sharp(f) (z) := \int_{w\in bD}  C^\sharp(w,z) f(w) d\lambda(w), \quad z\in bD
\end{align}
with the kernel
$$C^\sharp(w,z):={1\over g(w,z)^n},$$
and the reminder 
$$\EuScript R(f)(z):=\int_{w\in bD}R(w,z)f(w)dw.$$
Thus, if we write
$$\EuScript C(f)(z):=\int_{w\in bD}C(w,z)f(w)dw.$$
Then 
$$C(w,z)=C^\sharp(w,z)+ R(w,z).$$

\noindent Moreover, Lanzani--Stein pointed out that the kernel $C^\sharp(w,z)$  satisfies the standard size and smoothness conditions as in \eqref{gwz}.
Furthermore, concerning the function $g(w,z)$ in the size estimates, according to \cite[p. 139]{LS}, there exist $C_g, \tilde C_g>0$ such that 
\begin{align}\label{gd}
{\color{black} \tilde C_g{\tt d}(w,z)^{2}\leq |g(w,z)|\leq C_g {\tt d}(w,z)^{2}.}
 \end{align}
The kernel $R(w,z)$ of $\EuScript R$ satisfies the size and smoothness conditions as in \eqref{cr}.

 We now recall the BMO space on $bD$. 
 Consider $(bD, {\tt d}, d\lambda)$ as a space of homogeneous type with $bD$ compact. 
Then ${\rm BMO }(bD,d\lambda)$ is defined as the set of all $b\in L^1(bD,d\lambda)$ such that
$$ \|b\|_*:=\sup_{ z\in bD, r>0, B_r(z)\subset bD} {1\over\lambda(B_r(z))}\int_{B_r(z)} |b(w)-b_B|d\lambda(w)<\infty, $$
{with the balls $B_r(z)$ as in \eqref{ball} and where}
\begin{align}\label{fb}
b_B={1\over \lambda(B)}\int_B b(z)d\lambda(z).
\end{align}
And the norm is defined as 
$$\|b\|_{{\rm BMO }(bD,d\lambda)}:=\|b\|_*+ \|b\|_{L^1(bD,d\lambda)}. $$

 We now recall the VMO space on $bD$ (see for example \cite{KL1}).  
\begin{definition}\label{thm-cmo}
Let $f\in {\rm BMO}(bD, d\lambda)$. Then $f\in{\rm VMO}(bD, d\lambda)$ if and only if $f$ satisfies 
\begin{align}\label{vmpc}
\lim\limits_{a\rightarrow 0}\sup\limits_{B\subset bD:~r_B=a}{1\over \lambda(B)}\int_{B}|f(z)-f_B|d\lambda(z)=0,
\end{align}
where $r_B$ is the radius of $B$.
\end{definition}
We also let $BUC(bD)$ be the space of all bounded uniformly continuous functions on $bD$. 
Before we continue, we first point out that the Leray--Levi measure $d\lambda$ on $bD$ is a doubling measure, and  satisfies 
the condition (1.1) in \cite{KL2}. That is,
{
\begin{lemma}\label{measure lambda}
The Leray--Levi measure $d\lambda$ on $bD$ is doubling, i.e., there is a positive constant $C$ such that for all $x\in bD$ and $0<r\leq1$,
$$ 0<\lambda(B_{2r}(x))\leq C\lambda(B_{r}(x))<\infty. $$
Moreover, $\lambda$
satisfies the condition: there exist a constant $\epsilon_0\in(0,1)$ and a positive constant $C$ such that 
$$ \lambda( B_r(x)\backslash B_r(y) ) +  \lambda( B_r(y)\backslash B_r(x) ) \leq C\bigg( { {\tt d}(x,y) \over r}  \bigg)^{\epsilon_0}   $$
for all $x,y\in bD$ and ${\tt d}(x,y)\leq r\leq1$.
\end{lemma}
\begin{proof}
These facts can be verified directly from \eqref{lambdab}.
\end{proof}
}

Based on Lemma \ref{measure lambda}, we see that the following  fundamental lemma from \cite[Lemma 1.2]{KL2} can be applied to our setting $(bD, d\lambda)$, which we restate as follows.
\begin{lemma}\label{lemma 1 KL2}
Let $b\in{\rm VMO}(bD, d\lambda)$. Then for any $\xi>0$, there is a function $b_\xi\in BUC(bD)$ such that
\begin{align}\label{f_eta -f}
\|b_\xi -b\|_*<\xi.
\end{align}
Moreover, $b_\xi$ satisfies the following conditions: there is an $\epsilon\in (0,1)$ such that
\begin{align}\label{f_eta}
|b_\xi(w) -b_\xi(z)|<C_\xi {\tt d}(w,z)^\epsilon,\quad \forall w,z\in bD.
\end{align}
\end{lemma}
Moreover, based on Lemma \ref{measure lambda} we also have another fundamental lemma, whose proof follows from Lemma 1.3 in \cite{KL2}. For each $0<\eta<<1$, we let
$R^\eta(w,z)$ be a continuous extension of the kernel $R(w,z)$ of $\EuScript R$ from $bD\times bD\backslash \{(w,z): {\tt d}(w,z)<\eta\}$ to $bD\times bD$ such that
\begin{align*}
&R^\eta(w,z) =R(w,z),\quad {\rm if}\ \ {\tt d}(w,z)\geq\eta;\\
&|R^\eta(w,z)|\lesssim {1\over {\tt d}(w,z)^{2n-1}}, \quad {\rm if}\ \  {\tt d}(w,z)<\eta;\\
& R^\eta(w,z)=0, \quad {\rm if}\ \  {\tt d}(w,z)<\eta/c\ \ {\text {for some}}\ \ c>1.
\end{align*}
Now we let $\EuScript R^\eta$ be the integral operator associate to the kernel $R^\eta(w,z)$.
Then we have the following.
\begin{lemma}\label{lemma 2 KL2}
Let $b\in BUC(bD)$ satisfy 
\begin{align}\label{f_eta 1}
|b(w) -b(z)|<C_\eta {\tt d}(w,z)^\epsilon,\quad {\rm\ for\ some\ } C_\eta\geq1,\ \epsilon\in(0,1),\ \forall w,z\in bD.
\end{align}
Then 
$$ \|[b, \EuScript R]-[b, \EuScript R^\eta]\|_{L^2(bD,d\lambda)\to L^2(bD,d\lambda)}\to0 $$
as $\eta\to0$.
\end{lemma}
{We should mention that the kernel of $\EuScript R$ is not a standard kernel.}

\subsection{Characterisation of ${\rm BMO}(bD,d\lambda)$ via the Commutator $[b, \EuScript C]$ }
The maximal function $Mf$ is defined as
$$Mf(z)=\sup_{z\in B\subset bD}{1\over\lambda(B)}\int_B |f(w)|d\lambda(w).$$
The sharp function $f^\#$ is defined as
$$f^\#(z)=\sup_{z\in B\subset bD}{1\over\lambda(B)}\int_B |f(w)-f_B|d\lambda(w),$$
where $f_B$ is defined in \eqref{fb}.

{Note that from  Lemma \ref{measure lambda}, $d\lambda$ is a doubling measure, which implies that
$(bD,d\lambda)$ satisfies \cite[Theorems 1.4--1.6]{BC}. Hence, we can apply
 \cite[Theorems 1.4--1.6]{BC} to our setting as follows.}
\begin{lemma}[Maximal Inequality]\label{maximal}
For every $1< p\leq\infty$ there exists a constant $c(bD, p)$ such that for every $f\in L^p(bD, d\lambda)$, 
\begin{align*}
\left\|Mf\right\|_{L^p(bD, d\lambda)}\leq c\|f\|_{L^p(bD, d\lambda)}.
\end{align*}
\end{lemma}

\begin{lemma}[John-Nirenberg Inequality]\label{JN}
For every $1\leq p<\infty$ there exists a constant $c(bD, p)$ such that for every $f\in {\rm BMO}(bD,d\lambda)$,  every ball $B$,
\begin{align*}
\left({1\over \lambda(B)}\int_{B}\left |f(z)-f_B\right|^p d\lambda(z)\right)^{1\over p}\leq c\|f\|_{{\rm BMO}(bD,d\lambda)}.
\end{align*}
\end{lemma}

\begin{lemma}[Sharp Inequality]\label{lemma-sharp}
For every $1\leq p<\infty$ there exists a constant $c(bD, p)$ such that for every $f\in L^p(bD, d\lambda)$, 
\begin{align*}
\left\|f-{1\over \lambda(bD)}\int_{bD}f(w)d\lambda(w)\right\|_{L^p(bD, d\lambda)}\leq c\|f^\#\|_{L^p(bD, d\lambda)}.
\end{align*}
\end{lemma}
We note that in the unbounded domain we have $\|f\|_{L^p}\lesssim \|f^\#\|_{L^p}$, however, in the bounded domain, we will need to subtraction of the average of $f$ over the whole domain.

We also need the $L^p$ boundedness of $\EuScript C^\sharp$ and $\EuScript R$ which can be found in the proof of \cite[Theorem 7]{LS}.
\begin{lemma}\label{csharp-lp}
Suppose $1<p<\infty$ and $D\subset \mathbb C^n$, $n\geq 2$, is  a bounded domain whose boundary is of class $C^2$ and is strongly pseudoconvex. Then $\EuScript C^\sharp$ and $\EuScript R$ are  bounded operators on $L^p(bD, d\lambda)$.
\end{lemma}

We now prove the argument (1) in Theorem \ref{cauchy}.

\begin{proof}[Proof of $(1)$ in Theorem \ref{cauchy}]

\smallskip

{Necessity:}\ 
\smallskip

We first prove necessity, namely that $b\in {\rm BMO}(bD,d\lambda)$ implies the boundedness of $[b, \EuScript C]$. 
Since $b\in {\rm BMO}(bD,d\lambda)$,
without loss of generality we assume that $$\int_{bD} b(z)d\lambda(z) =0.$$
Otherwise we will just use $b(z)-{1\over \lambda(bD)}\int_{bD} b(w)d\lambda(w)$ to replace $b$.

We can write
$$[b,\EuScript C]=[b,\EuScript C^\sharp]+[b,\EuScript R].$$
From Lemma \ref{csharp-lp} we can see that $\EuScript C^\sharp$ is  bounded on $L^p(bD, d\lambda)$.
Since the kernel of $\EuScript C^\sharp$ is a standard kernel on $bD\times bD$, according to \cite[Theorerm 3.1]{KL1},  we can obtain that $[b, \EuScript C^\sharp]$ is bounded on  $L^p(bD, d\lambda)$ and
\begin{align*}
\|[b,\EuScript C^\sharp]\|_{L^p(bD, d\lambda)\rightarrow L^p(bD, d\lambda)}\lesssim \|b\|_{{\rm BMO} (bD,d\lambda)}.
\end{align*}
Thus, it suffices to show that 
\begin{align}\label{[b,R] bounded}
\|[b,\EuScript R]\|_{L^p(bD, d\lambda)\rightarrow L^p(bD, d\lambda)}\lesssim \|b\|_{{\rm BMO} (bD,d\lambda)}.
\end{align}

To see this, we first prove that for every $f\in L^p(bD, d\lambda)$,
\begin{align}\label{Tsharp}
\big\|\left([b,\EuScript R](f)\right)^\#\big\|_{L^p(bD, d\lambda)\rightarrow L^p(bD, d\lambda)}\lesssim \|b\|_{{\rm BMO} (bD,d\lambda)}\|f\|_{L^p(bD, d\lambda)}.
\end{align}

Since $bD$ is bounded, there exists $\overline C>0$ such that for any $B_r(z)\subset bD$ we have $r<\overline C$. 
For any $\tilde z \in bD$, let us fix a ball $B_r=B_r(z_0)\subset bD$ containing $\tilde z$,  and let $z$ be any point of $B_r$. 
Now take $j_0=\lfloor\log_2 {\overline C\over r}\rfloor+1$. Since ${\tt d}$ is a quasi-distance, 
 there exists $i_0\in\mathbb Z^+$, independent of $z$, $r$, such that ${\tt d}(w,z)>c_R r$ whenever $w\in bD\setminus B_{2^{i_0} r}$, where $c_R$ is in \eqref{cr}.
We then write
\begin{align*}
&[b,\EuScript R](f)(z)\\
&=b(z)\EuScript R(f)(z)-\EuScript R (bf)(z)\\
&=\big(b(z)-b_{B_r}\big)\EuScript R(f)(z)-\EuScript R\big((b-b_{B_r})f\chi_{ bD \cap B_{2^{i_0} r}} \big)(z)
-\EuScript R\big((b-b_{B_r})f\chi_{bD\setminus B_{2^{i_0} r}} \big)(z)\\
&=: I(z)+I\!I(z)+I\!I\!I(z).
\end{align*}

For $I$, by H\"older's inequality and the John-Nirenberg inequality (Lemma \ref{JN}), we have
\begin{align*}
&{1\over \lambda(B_r)}\int_{B_r}|I(z)-I_{B_r}|d\lambda(z)\\
&\leq{2\over \lambda(B_r)}\int_{B_r}|I(z)|d\lambda(z)\\
&={2\over \lambda(B_r)}\int_{B_r}\big|b(z)-b_{B_r} \big| \big|\EuScript R(f)(z)\big|d\lambda(z)\\
&\lesssim \left( {1\over \lambda(B_r)}\int_{B_r}\big|b(z)-b_{B_r} \big|^{s'}d\lambda(z)\right)^{{1\over s'}}
 \left( {1\over \lambda(B_r)}\int_{B_r}\big|\EuScript R(f)(z) \big|^{s}d\lambda(z)\right)^{1\over s}\\
 &\lesssim \|b\|_{{\rm BMO} (bD,d\lambda)} \big(M(|\EuScript R f|^s)(\tilde z) \big)^{1\over s},
\end{align*}
 where $1<s<p<\infty$.

For $I\!I$, since $\EuScript R$ is bounded on $L^q(bD, d\lambda)$, $1<q<\infty$, we have
\begin{align*}
&{1\over \lambda(B_r)}\int_{B_r}|I\!I(z)-I\!I_{B_r}|d\lambda(z)\\
&\leq{2\over \lambda(B_r)}\int_{B_r}|I\!I(z)|d\lambda(z)\\
&={2\over \lambda(B_r)}\int_{B_r}\big|\EuScript R\big((b-b_{B_r})f\chi_{bD\cap B_{2^{i_0} r}} \big)(z)\big|d\lambda(z)\\
&\lesssim \left( {1\over \lambda(B_r)}\int_{B_r}\big|\EuScript R\big((b-b_{B_r})f\chi_{bD\cap B_{2^{i_0} r}} \big)(z)\big|^qd\lambda(z) \right)^{1\over q}\\
&\lesssim \left( {1\over \lambda(B_r)}\int_{bD\cap B_{2^{i_0} r}}\big|b(z)-b_{B_r}\big|^q |f(z)|^qd\lambda(z) \right)^{1\over q}\\
&\lesssim \left( {1\over \lambda(B_r)}\int_{bD\cap B_{2^{i_0} r}}\big|b(z)-b_{B_r}\big|^{qv'} d\lambda(z) \right)^{1\over qv'}
\left( {1\over \lambda(B_r)}\int_{bD\cap B_{2^{i_0} r}} |f(z)|^{qv}d\lambda(z) \right)^{1\over qv}\\
&\lesssim \|b\|_{{\rm BMO} (bD,d\lambda)} \big(M(| f|^\beta)(\tilde z) \big)^{1\over \beta},
\end{align*}
 where we have chosen $q,v\in (1,\infty)$ such that $1<qv<p<\infty$ and have set $\beta:=qv$.

To estimate $I\!I\!I$, observe that if $i_0\geq j_0$, then $bD\setminus B_{2^{i_0} r}=\emptyset$ and $| I\!I\!I(z)-I\!I\!I(z_0)|=0$.
If $i_0<j_0$, then we have
\begin{align*}
&\left| I\!I\!I(z)-I\!I\!I(z_0)\right|\\
&=\big|\EuScript R\big((b-b_{B_r})f\chi_{bD\setminus B_{2^{i_0} r}} \big)(z)
-\EuScript R\big((b-b_{B_r})f\chi_{bD\setminus B_{2^{i_0} r}} \big)(z_0)\big|\\
&\leq \int_{bD\setminus B_{2^{i_0} r}}\big| R(w,z)-R(w,z_0) \big| \left|b(w)-b_{B_r} \right| |f(w)|d\lambda(w)\\
&\leq {\tt d}(z,z_0) \int_{bD\setminus B_{2^{i_0} r}}{1\over {\tt d}(w, z_0)^{2n}}\left|b(w)-b_{B_r} \right| |f(w)|d\lambda(w)\\
&\leq  r \left( \int_{bD\setminus B_{2^{i_0} r}}{1\over {\tt d}(w, z_0)^{2n}}\left|b(w)-b_{B_r} \right|^{s'}d\lambda(w)\right)^{1\over s'} \left(\int_{bD\setminus B_{2^{i_0} r}}{1\over {\tt d}(w, z_0)^{2n}}|f(w)|^sd\lambda(w) \right)^{1\over s},
\end{align*}
 where $1<s<p<\infty$.
%%%
 Since $bD$ is bounded,  we can obtain
\begin{align*}
\int_{bD\setminus B_{2^{i_0} r}}{1\over {\tt d}(w, z_0)^{2n}}|f(w)|^sd\lambda(w) 
&\leq\sum_{j=i_0}^{j_0}\int_{2^j r\leq d(w, z_0)\leq 2^{j+1} r}{1\over {\tt d}(w, z_0)^{2n}}|f(w)|^sd\lambda(w) \\
&\leq \sum_{j=i_0}^{j_0}{1\over (2^j r)^{2n}}
\int_{d(w, z_0)\leq 2^{j+1} r}|f(w)|^sd\lambda(w)\\
&\lesssim\sum_{j=i_0}^{j_0}  {1\over \lambda(B_{2^{j+1} r})}\int_{B_{2^{j+1} r}}|f(w)|^sd\lambda(w)\\
&\lesssim  j_0 M(|f|^s)(\tilde z).
\end{align*}
Similarly, by the John--Nirenberg inequality, we have
\begin{align*}
 \int_{bD\setminus B_{2^{i_0} r}}{1\over {\tt d}(w, z_0)^{2n}}\left|b(w)-b_{B_r} \right|^{s'}d\lambda(w)
 &\lesssim\sum_{j=i_0}^{j_0}  {1\over \lambda(B_{2^{j+1} r})}\int_{B_{2^{j+1} r}} \left|b(w)-b_{B_r} \right|^{s'}d\lambda(w)\\
&\lesssim j_0\|b\|^{s'}_{{\rm BMO} (bD,d\lambda)}.
\end{align*}
Thus,
\begin{align*}
\left|I\!I\!I(z)-I\!I\!I(z_0)\right|\lesssim  r j_0  \|b\|_{{\rm BMO} (bD,d\lambda)}\left(M(|f|^s)(\tilde z)\right)^{1\over s}.
\end{align*}
Therefore,
\begin{align*}
{1\over \lambda(B_r)}\int_{B_r}\big| I\!I\!I(z)-I\!I\!I_{B_r}\big|d\lambda(z)
&\leq {2\over \lambda(B_r)}\int_{B_r}\big| I\!I\!I(z)-I\!I\!I(z_0)\big|d\lambda(z)\\
&\lesssim r  j_0  \|b\|_{{\rm BMO} (bD,d\lambda)}\left(M(|f|^s)(\tilde z)\right)^{1\over s}\\
&\lesssim r  \big(\log_2\big( {\overline C\over r} \big)+1\big) \|b\|_{{\rm BMO} (bD,d\lambda)}\left(M(|f|^s)(\tilde z)\right)^{1\over s}\\
&\lesssim \|b\|_{{\rm BMO} (bD,d\lambda)}\left(M(|f|^s)(\tilde z)\right)^{1\over s},
\end{align*}
where the last inequality comes from the fact that $ r \log_2( {\overline C\over r})$ is uniformly bounded.

By the above estimates we  obtain that
\begin{align*}
|([b,\EuScript R]f)^\# (\tilde z)|
&\lesssim
\|b\|_{{\rm BMO} (bD,d\lambda)}\left( \big(M(|\EuScript R f|^s)(\tilde z) \big)^{1\over s}+ \big(M(| f|^\beta)(\tilde z) \big)^{1\over \beta}
+\left(M(|f|^s)(\tilde z)\right)^{1\over s}\right),
\end{align*}
which, together with  Lemma \ref{maximal} and the fact that $\EuScript R$ is bounded on $L^p(bD, d\lambda)$, implies that 
 \eqref{Tsharp} holds.

Based on \eqref{Tsharp}, we now prove \eqref{[b,R] bounded}. We note that in the unbounded domain we have that $\|g\|_{L^p}\lesssim \|g^\#\|_{L^p}$, however, in this bounded domain, we will need a subtraction of the average of $g$ over the whole domain on the left-hand side of this inequality (see Lemma \ref{lemma-sharp}). 

Thus, we have that
\begin{align}\label{[b,R] Lp}
\left\|[b,\EuScript R]f\right\|_{L^p(bD, d\lambda)}&\leq \Big\|[b,\EuScript R]f - {1\over \lambda(bD)}\int_{bD} [b,\EuScript R]f(z) d\lambda(z)\Big\|_{L^p(bD, d\lambda)}\\
&\qquad+ \Big\| {1\over \lambda(bD)}\int_{bD} [b,\EuScript R]f(z) d\lambda(z)\Big\|_{L^p(bD, d\lambda)} \nonumber\\
&\leq \Big\| ([b,\EuScript R]f)^\# \Big\|_{L^p(bD, d\lambda)} + {1\over \lambda(bD)}\int_{bD} \big|[b,\EuScript R]f(z)\big| d\lambda(z)\  \lambda(bD)^{1\over p} \nonumber\\
&\lesssim\|b\|_{ {\rm BMO} (bD,d\lambda)}\|f\|_{L^p(bD, d\lambda)}+  {1\over \lambda(bD)^{1\over p'}}\int_{bD} \big|[b,\EuScript R]f(z)\big| d\lambda(z),\nonumber
\end{align}
where the second inequality follows from Lemma \ref{lemma-sharp} and the  last inequality 
follows from \eqref{Tsharp}. Now it suffice to show that 
\begin{align}\label{[b,R] L1}
\int_{bD} \big|[b,\EuScript R]f(z)\big| d\lambda(z)\lesssim  \lambda(bD)^{1\over p' }\| f\|_{L^p(bD,d\lambda)}  \|b\|_{{\rm BMO}(bD,d\lambda)}.
\end{align}

Note that
\begin{align*}
\int_{bD} \big|[b,\EuScript R]f(z)\big| d\lambda(z) 
&\leq  \int_{bD} \big|\EuScript R(f)(z) \,b(z)\big|\, d\lambda(z)   + \int_{bD} \big|\EuScript R(bf)(z)\big|  d\lambda(z)\\
&=: \mathfrak A_1+ \mathfrak A_2. 
\end{align*}
For the term $ \mathfrak A_1$, by H\"older's inequality, the John--Nirenberg inequality and recalling that $b_{bD}=0$, we have that
\begin{align*}
 \mathfrak A_1&\leq  \| \EuScript R(f)\|_{L^p(bD,d\lambda)} \|b\|_{L^{p'}(bD,d\lambda)}
= \| \EuScript R(f)\|_{L^p(bD,d\lambda)} \|b - b_{bD}\|_{L^{p'}(bD,d\lambda)}\\
&\lesssim \| f\|_{L^p(bD,d\lambda)} \cdot \lambda(bD)^{1\over p' }\cdot \bigg( {1\over\lambda(bD)}\int_{bD} | b(z)-b_{bD} |^{p'}d\lambda(z) \bigg)^{1\over p'}  \\
&\lesssim  \lambda(bD)^{1\over p' }\| f\|_{L^p(bD,d\lambda)}  \|b\|_{{\rm BMO}(bD,d\lambda)}.
\end{align*}
For the term $ \mathfrak A_2$, by H\"older's inequality, the John--Nirenberg inequality and recalling that $b_{bD}=0$, we have that
\begin{align*}
 \mathfrak A_2&\leq \lambda(bD)^{1\over \gamma'}  \| \EuScript R(bf)\|_{L^\gamma(bD,d\lambda)} \\
& \lesssim\lambda(bD)^{1\over \gamma'} \| bf\|_{L^\gamma(bD,d\lambda)} \\
&\lesssim \lambda(bD)^{1\over \gamma'} \| f\|_{L^p(bD,d\lambda)} \cdot \| b\|_{L^{\mu}(bD,d\lambda)}  \\
&= c\lambda(bD)^{1\over \gamma'} \| f\|_{L^p(bD,d\lambda)} \cdot \|b - b_{bD}\|_{L^{\mu}(bD,d\lambda)} \\
&\lesssim \lambda(bD)^{1\over \gamma'} \| f\|_{L^p(bD,d\lambda)} \lambda(bD)^{1\over \mu } \|b\|_{{\rm BMO}(bD,d\lambda)}\\
&\lesssim \lambda(bD)^{1\over p'} \| f\|_{L^p(bD,d\lambda)}  \|b\|_{{\rm BMO}(bD,d\lambda)},
\end{align*}
where we have chosen $\gamma, \mu>1$ satisfying
$ {1\over \gamma}  = {1\over p } +{1\over \mu}.$
Therefore, \eqref{[b,R] L1} holds, which, together with \eqref{[b,R] Lp}, implies that 
\eqref{[b,R] bounded} holds. Hence, the proof of the necessity part is complete.

\smallskip
{Sufficiency:}\ 
\smallskip

We next turn to proving the sufficient condition, namely that if $[b,\EuScript C]$ is bounded, then $b\in {\rm BMO}(bD,d\lambda)$.  Suppose $1<p<\infty$. Assume that $b$ is in $L^1(bD,d\lambda)$ and that $\left\|[b,\EuScript C]\right\|_{L^p(bD, d\lambda)\to L^p(bD, d\lambda)} <\infty$. 

We now denote by $ C^\sharp_1(w,z)$ and $ C^\sharp_2(w,z)$
the real and imaginary parts of $ C^\sharp(w,z)$, respectively.
And denote by $ R_1(w,z)$ and $ R_2(w,z)$
the real and imaginary parts of $ R(w,z)$, respectively.
Since $ C^\sharp(w,z)$ is a standard Calder\'on--Zygmund kernel (we refer to \eqref{gwz}), it is direct to verify the following argument: let
\begin{align}\label{gamma0}
\gamma_0:={{3 \over  4^{n+2}10 C_d^{4n+2} C_g C_R \max\{ {5A_1A_2 (2C_d)^{2n+1}C_g }, \sqrt2C_dA_2 + \sqrt2C_d^2\}},}
\end{align}
then for every ball $B=B_r(z_0)\subset bD$ with $r<\gamma_0$, there exists another ball
$\tilde B = B_r(w_0)\subset bD$ with  $A_3r \leq {\tt d}(w_0,z_0) \leq (A_3+1)r$, where
$A_3 ={\max\{ 5\sqrt2 A_1A_2 (2C_d)^{2n+1} C_g,  2C_dA_2 + 2C_d^2\}}$, such that
at least one of the following properties holds:

$\mathfrak a)$ for every $z\in B$ and $w\in \tilde B$, $ C^\sharp_1(w,z)$ does not change sign and $$| C^\sharp_1(w,z)|\geq { 3\sqrt2\over 10(2C_d^2)^{2n} C_g {\tt d}(w,z)^{2n} };$$

$\mathfrak b)$ for every $z\in B$ and $w\in \tilde B$, $ C^\sharp_2(w,z)$ does not change sign and $$| C^\sharp_2(w,z)|\geq { 3\sqrt2\over 10(2C_d^2)^{2n} C_g {\tt d}(w,z)^{2n} }.$$

Here  the constants $A_1,C_g, C_R$ are defined in \eqref{gwz}, \eqref{gd}, \eqref{cr}, respectively, and the constants 
$A_2$ and $C_d$ are from \eqref{metric d}.

Without lost of generality, we assume that the property $\mathfrak a)$ holds. Then combining with the size estimate of $R(w,z)$ in \eqref{cr}, we obtain that for every $z\in B$ and $w\in \tilde B$, $ C^\sharp_1(w,z)+ R_1(w,z)$
does not change sign and that 
\begin{align}\label{kernel lower bound C1}
| C^\sharp_1(w,z)+ R_1(w,z)|\geq { 3\sqrt2 \over 20 (2C_d^2)^{2n} C_g {\tt d}(w,z)^{2n} }.
\end{align}

We test the ${\rm BMO}(bD,d\lambda)$ condition on the case of balls with big radius and small radius.

Case 1: In this case we work with balls with a large radius, $r\geq \gamma_0$. 

By \eqref{lambdab} and by the fact that 
$\lambda(B)\geq \lambda( B_{\gamma_0}(z_0)) \approx \gamma_0^{2n}$, we obtain that
\begin{align*}%\label{bmo norm1}
{1\over \lambda(B)} \int_B|b(w)-b_B|d\lambda(w)\lesssim \gamma_0^{-2n} \|b\|_{L^1(bD, d\lambda)}.
\end{align*}

Case 2: In this case we work with balls with a small radius, $r<\gamma_0$.

We aim to prove that for every fixed ball $B=B_r(z_0)\subset bD$ with radius $r<\gamma_0$,
\begin{align*}
{1\over \lambda(B)} \int_B|b(w)-b_B|d\lambda(w)\lesssim \left\|[b,\EuScript C]\right\|_{L^p(bD, d\lambda)\to L^p(bD, d\lambda)}.
\end{align*}

Now let $\tilde B=B_r(w_0)$ be the ball chosen as above, and let $m_b(\tilde B)$ be the median value of $b$ on the ball $\tilde B$ with respect to the measure $d\lambda$ defined as follows:
$m_b(\tilde B)$ is a real number that satisfies simultaneously
$$ \lambda(\{z\in \tilde B: b(z)>m_b(\tilde B)\})\leq {1\over2}\lambda(\tilde B)\quad {\rm and}\quad \lambda(\{z\in \tilde B: b(z)<m_b(\tilde B)\}) \leq {1\over2}\lambda(\tilde B).$$
Then, following the idea in \cite[Proposition 3.1]{LOR} (see also \cite{TYY}), by the definition of median value, we define
$F_1:=\{ w\in \tilde B: b(w)\leq m_b(\tilde B) \}$ and
$F_2:=\{ w\in \tilde B: b(w)\geq m_b(\tilde B) \}$. Then it is direct that $\tilde B=F_1\cup F_2$, and moreover, from the definition of $m_b(\tilde B)$, we see that
\begin{align}\label{f1f2-1}
\lambda(F_i) \geq{1\over 2} \lambda(\tilde B),\quad i=1,2.
\end{align}

Next we define
\begin{align*}%\label{EF12}
E_1=\{ z\in B: b(z)\geq m_b(\tilde B) \}\quad{\rm and}\quad
E_2=\{ z\in B: b(z)< m_b(\tilde B) \},
\end{align*}
then $B=E_1\cup E_2$ and $E_1\cap E_2=\emptyset$. 
Then it is clear that 
\begin{equation}\begin{split}\label{bx-by0-1}
b(z)-b(w) &\geq 0, \quad (z,w)\in E_1\times F_1\\
b(z)-b(w) &< 0, \quad (z,w)\in E_2\times F_2.
\end{split}\end{equation}
And for $(z,w)$ in $(E_1\times F_1 )\cup (E_2\times F_2)$, we have 
\begin{align}\label{bx-by-1}
\nonumber|b(z)-b(w)| 
&= |b(z)-m_b(\tilde B) +m_b(\tilde B) -b(w)| \\
&=|b(z)-m_b(\tilde B)| +|m_b(\tilde B) -b(w)| \\
&\nonumber\geq |b(z)-m_b(\tilde B)|. 
\end{align}

Therefore, from \eqref{f1f2-1}, \eqref{kernel lower bound C1} and \eqref{bx-by-1} 
we obtain that 
\begin{align*}
&{1\over \lambda(B)}\int_{E_{1}}\big|b(z)-m_b(\tilde B)\big|d\lambda(z)\\
&\lesssim 
{1\over \lambda(B)}{\lambda{(F_1)}\over \lambda(B)}\int_{E_{1}}\big|b(z)-m_b(\tilde B)\big|d\lambda(z)\\
&\lesssim
{1\over \lambda(B)}\int_{E_{1}}\int_{F_{1}}{1\over {\tt d}(w,z)^{2n}}\big|b(z)-b(w)\big|d\lambda(w)d\lambda(z)\\
&\lesssim
{1\over \lambda(B)}\int_{E_{1}}\int_{F_{1}}| C^\sharp_1(w,z)+ R_1(w,z)|\big(b(z)-b(w)\big)d\lambda(w)d\lambda(z)\\
&\lesssim
{1\over \lambda(B)}\int_{E_{1}}\bigg|\int_{F_{1}}C(w,z)\big(b(z)-b(w)\big)d\lambda(w)\bigg|d\lambda(z)\\
&=
{1\over \lambda(B)}\int_{E_{1}}\left|[b, \EuScript C](\chi_{F_1})(z)\right|d\lambda(z),
\end{align*}
where the last but second inequality follows from the fact that $C^\sharp_1(w,z)+ R_1(w,z)$ is the real part of 
$C(w,z)$.

Then, by using H\"older's inequality and the fact that $\left\|[b,\EuScript C]\right\|_{L^p(bD, d\lambda)\to L^p(bD, d\lambda)} <\infty$, we further obtain 
\begin{align*}
&{1\over \lambda(B)}\int_{E_{1}}\big|b(z)-m_b(\tilde B)\big|d\lambda(z)\\
&\lesssim
{1\over \lambda(B)}\left(\lambda(E_1)\right)^{1\over p'}\bigg(\int_{E_{1}}\left|[b, \EuScript C](\chi_{F_1})(z)\right|^pd\lambda(z)\bigg)^{1\over p}\\
&\lesssim
 {1\over \lambda(B)}\left(\lambda(E_1)\right)^{1\over p'}\|[b, \EuScript C]\chi_{F_1}\|_{L^p(bD,d\lambda)}\\
 &\lesssim
  {1\over \lambda(B)}\left(\lambda(E_1)\right)^{1\over p'}\left(\lambda(F_1)\right)^{1\over p}
\|[b, \EuScript C]\|_{L^p(bD,d\lambda)\to L^p(bD,d\lambda)}\\
&\lesssim
 {1\over \lambda(B)}\left(\lambda(E_1)+\lambda(F_1)\right)
\|[b, \EuScript C]\|_{L^p(bD,d\lambda)\to L^p(bD,d\lambda)},
\end{align*}
where the last inequality follows from a direct calculation:
if $\lambda\big(E_1\big) \geq \lambda\big( F_1\big)$, then 
$$\big(\lambda\big(E_1\big)\big)^{1\over p'}\big(\lambda\big( F_1\big)\big)^{1\over p} \leq \big(\lambda\big(E_1\big)\big)^{1\over p'}\big( \lambda\big(E_1\big) \big)^{1\over p}\leq \lambda\big(E_1\big);$$
if $\lambda\big(E_1\big) \leq \lambda\big( F_1\big)$, then 
$$\big(\lambda\big(E_1\big)\big)^{1\over p'}\big(\lambda\big( F_1\big)\big)^{1\over p} \leq \big(\lambda\big(F_1\big)\big)^{1\over p'}\big( \lambda\big(F_1\big) \big)^{1\over p}\leq \lambda\big(F_1\big).$$

Similarly, we can obtain that 
\begin{align*}
{1\over \lambda(B)}\int_{E_{2}}\big|b(z)-m_b(\tilde B)\big|d\lambda(z)
\lesssim 
{1\over \lambda(B)}\left(\lambda(E_2)+\lambda(F_2)\right)
\|[b, \EuScript C]\|_{L^p(bD,d\lambda)\to L^p(bD,d\lambda)}.
\end{align*}

Consequently,
\begin{align*}
&{1\over \lambda(B)}\int_{B}\big|b(z)-m_b(\tilde B)\big|d\lambda(z)\\
&={1\over \lambda(B)}\int_{E_1}\big|b(z)-m_b(\tilde B)\big|d\lambda(z)
+{1\over \lambda(B)}\int_{E_2}\big|b(z)-m_b(\tilde B)\big|d\lambda(z)\\
&\lesssim 
{1\over \lambda(B)}\left(\lambda(E_1)+\lambda(F_1)+\lambda(E_2)+\lambda(F_2)\right)
\|[b, \EuScript C]\|_{L^p(bD,d\lambda)\to L^p(bD,d\lambda)}\\
&\lesssim
\|[b, \EuScript C]\|_{L^p(bD,d\lambda)\to L^p(bD,d\lambda)}.
\end{align*}

Therefore,
\begin{align*}%\label{bmo norm2}
{1\over \lambda(B)}\int_{B}\big|b(z)-b_B\big|d\lambda(z)
\leq {2\over \lambda(B)}\int_{B}\big|b(z)-m_b(\tilde B)\big|d\lambda(z)
\lesssim \big\| [b,\EuScript C] \big\|_{L^p(bD, d\lambda)\rightarrow L^p(bD, d\lambda)}.
\end{align*}
This finishes the proof of (1) in Theorem \ref{cauchy}.
\end{proof}

\subsection{Characterisation of ${\rm VMO}(bD,d\lambda)$ via the Commutator $[b, \EuScript C]$}

We now prove the argument (2) in Theorem \ref{cauchy}.

\medskip

\begin{proof}[Proof of $(2)$ in Theorem \ref{cauchy}]
{\bf Sufficient condition:}
Assume that $1<p<\infty$ and that $[b, \EuScript C]$ is compact on $L^p(bD, d\lambda)$, then $[b,\EuScript C]$ is bounded on $L^p(bD, d\lambda)$. By the argument (1) in Theorem \ref{cauchy}, we have $b\in {\rm BMO}(bD, d\lambda)$. Without loss of generality, we may assume that $\|b\|_{  {\rm BMO}(bD, d\lambda) }=1$. 

\textcolor{black}{
To show $b\in {\rm VMO}(bD, d\lambda)$, we  seek a contradiction.  In its simplest form, the contradiction is that there is no bounded operator $T \;:\; \ell^{p} (\mathbb N) \to \ell^{p} (\mathbb N)$ with $Te_{j } = T e_{k} \neq 0$ for all $j,k\in \mathbb N$.  Here, $e_{j}$ is the 
standard basis for $\ell^{p} (\mathbb N)$.  }

\textcolor{black}{
The main step is to construct the approximates to a standard basis in $\ell^{p}$, namely a sequence of functions $\{g_{j}\}$ such that 
$ \lVert g_{j } \rVert _{L^p(bD, d \lambda } \simeq 1$,  and for a nonzero $ \phi$,   we have $\lVert \phi - [b, \EuScript C] g_{j}\rVert 
 _{L^p(bD, d \lambda) } <2^{-j}$.  } 

Suppose that $b\notin {\rm VMO}(bD, d\lambda)$, then there exist $\delta_0>0$ and a sequence $\{B_j\}_{j=1}^\infty:=\{B_{r_j}(z_j)\}_{j=1}^\infty$ of balls such that 
\begin{align}\label{delta0}
{1\over \lambda(B_j)} \int_{B_j} |b(z)-b_{B_j}|d\lambda(z)\geq \delta_0. 
\end{align}

Without lost of generality, we assume that for all $j$, $r_j<\gamma_0$, where $\gamma_0$ is 
as in \eqref{gamma0}.

Now choose a subsequence $\{B_{j_i}\}$ of $\{B_j\}$ such that
\begin{align}\label{ratio1}
r_{j_{i+1}} \leq{1\over 4c_\lambda} r_{j_{i}},
\end{align}
where $c_\lambda$ is the constant as in \eqref{lambdab}.

For the sake of simplicity we drop the subscript $i$, i.e., we still denote $\{B_{j_i}\}$ by $\{B_{j}\}$.

Following the proof of (1) of Theorem \ref{cauchy}, for each such $B_j$, we can choose a corresponding $\tilde B_j$.
Now let $m_b(\tilde B_j)$ be the median value of $b$ on the ball $\tilde B_j$ with respect to the measure $d\lambda$.
Then,  by the definition of median value, we can
find disjoint subsets $F_{j,1},F_{j,2}\subset \tilde B_j$ such that
\begin{align*}
F_{j,1}\subset\{ w\in \tilde B_j: b(w)\leq m_b(\tilde B_j) \},\quad
 F_{j,2}\subset\{ w\in \tilde B_j: b(w)\geq m_b(\tilde B_j) \},
\end{align*}
and  
\begin{align}\label{f1f2-1 1}
\lambda(F_{j,1}) = \lambda(F_{j,2}) ={\lambda(\tilde B_j)\over 2} .
\end{align}

Next we define
\begin{align*}
 E_{j,1}=\{ z\in B: b(z)\geq m_b(\tilde B_j) \},\quad
E_{j,2}=\{ z\in B: b(z)< m_b(\tilde B_j) \},
\end{align*}
then $B_j=E_{j,1}\cup E_{j,2}$ and $E_{j,1}\cap E_{j,2}=\emptyset$. 
Then it is clear that 
\begin{equation}\begin{split}\label{bx-by0-1 1}
b(z)-b(w) &\geq 0, \quad (z,w)\in E_{j,1}\times F_{j,1},\\
b(z)-b(w) &< 0, \quad (z,w)\in E_{j,2}\times F_{j,2}.
\end{split}\end{equation}
And for $(z,w)$ in $(E_{j,1}\times F_{j,1} )\cup (E_{j,2}\times F_{j,2})$, we have 
\begin{align}\label{bx-by-1 1}
\nonumber|b(z)-b(w)| 
&= |b(z)-m_b(\tilde B_j) +m_b(\tilde B_j) -b(w)| \\
&=|b(z)-m_b(\tilde B_j)| +|m_b(\tilde B_j) -b(w)| \\
&\nonumber\geq |b(z)-m_b(\tilde B_j)|. 
\end{align}

We now consider 
$$\widetilde F_{j,1}:= F_{j,1}\backslash \bigcup_{\ell=j+1}^\infty \tilde B_\ell\quad{\rm and}\quad \widetilde F_{j,2}:= F_{j,2}\backslash \bigcup_{\ell=j+1}^\infty \tilde B_\ell,\quad {\rm for}\ j=1,2,\ldots.$$
Then, based on the decay condition of the radius $\{r_j\}$, we obtain that
for each $j$,
\begin{align}\label{Fj1}
 \lambda(\widetilde F_{j,1}) &\geq \lambda(F_{j,1})- \lambda\Big( \bigcup_{\ell=j+1}^\infty \tilde B_\ell\Big) \geq
{1\over 2} \lambda(\tilde B_j)-\sum_{\ell=j+1}^\infty \lambda\big(  \tilde B_\ell\big)\nonumber\\
&\geq {1\over 2} \lambda(\tilde B_j)- {c_\lambda^2\over (4c_\lambda)^{2n}-1}\lambda(\tilde B_j)\geq {1\over 4} \lambda(\tilde B_j).
\end{align}

Now for each $j$, we have that
\begin{align*}
&{1\over \lambda(B_j)} \int_{B_j} |b(z)-b_{B_j}|d\lambda(z)\\
&\leq{2\over \lambda(B_j)}\int_{B_{j}}\big|b(z)-m_b(\tilde B_j)\big|d\lambda(z)\\
&= {2\over \lambda(B_j)}\int_{E_{j,1}}\big|b(z)-m_b(\tilde B_j)\big|d\lambda(z) + {2\over \lambda(B_j)}\int_{E_{j,2}}\big|b(z)-m_b(\tilde B_j)\big|d\lambda(z).
\end{align*}
Thus, combing with \eqref{delta0} and the above inequalities, we obtain that as least one of the following inequalities holds:
\begin{align*}
{2\over \lambda(B_j)}\int_{E_{j,1}}\big|b(z)-m_b(\tilde B_j)\big|d\lambda(z) \geq {\delta_0\over2},\quad 
{2\over \lambda(B_j)}\int_{E_{j,2}}\big|b(z)-m_b(\tilde B_j)\big|d\lambda(z) \geq {\delta_0\over2}.
\end{align*}

We may assume that the first one holds, i.e., 
\begin{align*}
{2\over \lambda(B_j)}\int_{E_{j,1}}\big|b(z)-m_b(\tilde B_j)\big|d\lambda(z) \geq {\delta_0\over2}.
\end{align*}

Therefore, for each $j$, from \eqref{f1f2-1 1} and \eqref{bx-by-1 1} we obtain that 
\begin{align*}
{\delta_0\over4}&\leq{1\over \lambda(B_j)}\int_{E_{j,1}}\big|b(z)-m_b(\tilde B_j)\big|d\lambda(z)\\
&\lesssim 
{1\over \lambda(B_j)}{\lambda{(\widetilde F_{j,1})}\over \lambda(B_j)}\int_{E_{j,1}}\big|b(z)-m_b(\tilde B_j)\big|d\lambda(z)\\
&\lesssim
{1\over \lambda(B_j)}\int_{E_{j,1}}\int_{\widetilde F_{j,1}}{1\over {\tt d}(w,z)^{2n}}\big|b(z)-b(w)\big|d\lambda(w)d\lambda(z).
\end{align*}
Following the same estimate as in the proof of (1) of Theorem \ref{cauchy}, we obtain that 
\begin{align*}
{1\over \lambda(B_j)}\int_{E_{j,1}}\big|b(z)-m_b(\tilde B_j)\big|d\lambda(z)
&\lesssim
{1\over \lambda(B_j)}\int_{E_{j,1}}\left|[b, \EuScript C](\chi_{\widetilde F_{j,1}})(z)\right|d\lambda(z)\\
&=
{1\over \lambda(B_j)^{1\over p'}}\int_{E_{j,1}}\left|[b, \EuScript C]\bigg({\chi_{\widetilde F_{j,1}} \over \lambda(B_j)^{1\over p}}\bigg)(z)\right|d\lambda(z),
\end{align*}
where in the last equality, we use $p'$ to denote the conjugate index of $p$. 

Next, by using H\"older's
inequality we further have
\begin{align*}
{1\over \lambda(B_j)}\int_{E_{j,1}}\big|b(z)-m_b(\tilde B_j)\big|d\lambda(z)
&\lesssim
{1\over \lambda(B_j)^{1\over p'}} \lambda(E_{j,1})^{1\over p'} \bigg( \int_{bD}\big|[b, \EuScript C](f_j)(z)\big|^pd\lambda(z) \bigg)^{1\over p}\\
&\lesssim
\bigg( \int_{bD}\big|[b, \EuScript C](f_j)(z)\big|^pd\lambda(z) \bigg)^{1\over p},
\end{align*}
where in the above inequalities 
we denote
$$ f_j := {\chi_{\widetilde F_{j,1}} \over \lambda(B_j)^{1\over p}}.$$

Thus, combining the above estimates we have that  
$$ 0<\delta_0  \lesssim
\bigg( \int_{bD}\big|[b, \EuScript C](f_j)(z)\big|^pd\lambda(z) \bigg)^{1\over p}.$$

Then, from \eqref{Fj1}, we obtain that
$$ {1\over4^{1\over p}}\lesssim \|f_j\|_{L^p(bD,d\lambda)}\lesssim 1.  $$
Thus, it is direct to see that $\{f_j\}_j$ is a bounded sequence in $L^p(bD,d\lambda)$ with a uniform $L^p(bD,d\lambda)$-lower bound away from zero.

Since $[b, \EuScript C]$ is compact, we obtain that the  sequence
$ \{[b, \EuScript C](f_j)\}_j $
has a convergent subsequence, denoted by
$$ \{[b, \EuScript C](f_{j_i})\}_{j_i}. $$
We denote the limit function by $g_0$, i.e.,
$$ [b, \EuScript C](f_{j_i})\to g_0 \quad{\rm in\ }  L^p(bD,d\lambda), \quad{\rm as\ } i\to\infty.$$
Moreover,  $g_{0} \neq 0$. 

After taking a further subsequence, labeled $g_{j}$, we have 
\begin{itemize} 
\item  $\lVert g_{j} \rVert _{L^{p} (bD, d \lambda )}  \simeq 1$, 
\item  $g_{j}$ are disjointly supported, 
\item and $\lVert g_{0}  -  [b, \EuScript C] g_{j} \rVert _{L^{p} (bD, d \lambda )}  < 2^{-j}$.  
\end{itemize}

\textcolor{black}{
Take $a_{j } = j ^{ \frac p {p+1}}$, so that $ \{a_{j}\} \in \ell^{p} \setminus \ell^{1}$.   It is immediate that $ \gamma = \sum_{j} a_{j} g_{j} \in 
L^{p} (bD, d \lambda )$, hence $ [b, \EuScript C] \gamma \in L^{p} (bD, d \lambda )$.  But, $ g_{0} \sum_{j} a_{j} \equiv \infty$, and yet 
 \begin{align*}
\Bigl\lVert  g_{0} \sum_{j} a_{j} \Bigr\rVert _{L^{p} (bD, d \lambda )} 
& \leq \lVert  [b, \EuScript C] \gamma  \rVert _{L^{p} (bD, d \lambda )}  
+ \sum_{j} a_{j}  \lVert   g_{0} -  [b, \EuScript C] g_{j}  \rVert _{L^{p} (bD, d \lambda )}  < \infty.
\end{align*}
This contradiction shows that $b\in {\rm VMO}(bD, d\lambda)$.
}

 \iffalse   
 
Note that all the functions $f_j$ are pairwise disjointly supported. We then take non-negative numerical sequence $\{a_j\}$ with 
$$ \|\{a_i\}\|_{\ell^p}<\infty \quad {\rm but}\quad   \|\{a_i\}\|_{\ell^1}=\infty.  $$
Then there holds
\begin{align*}
&\sum_{i=1}^\infty\bigg( a_i\|f_0\|_{L^p(bD,d\lambda)} -  a_i \|f_0- [b, \EuScript C](f_{j_i})\|_{L^p(bD,d\lambda)} \bigg)\\
&\leq \bigg\| \sum_{i=1}^\infty a_i [b, \EuScript C](f_{j_i}) \bigg\|_{L^p(bD,d\lambda)}=
 \bigg\|  [b, \EuScript C]\Big(\sum_{i=1}^\infty a_i f_{j_i}\Big) \bigg\|_{L^p(bD,d\lambda)}\lesssim  
\bigg\| \sum_{i=1}^\infty a_i f_{j_i} \bigg\|_{L^p(bD,d\lambda)}\\
&\lesssim \|\{a_i\}\|_{\ell^p}.
\end{align*}
Above, we use the triangle inequality, and then the upper bound on the norm of the commutator, and then the disjoint support condition. But the left-hand side is infinite by design because
$$  \sum_{i=1}^\infty  a_i\|f_0\|_{L^p(bD,d\lambda)}\gtrsim  \sum_{i=1}^\infty  a_i\delta_0=+\infty$$
and
\begin{align*}
\sum_{i=1}^\infty  a_i \|f_0- [b, \EuScript C](f_{j_i})\|_{L^p(bD,d\lambda)} &\leq \bigg(\sum_{i=1}^\infty  a_i^p \bigg)^{1\over p}
\bigg(\sum_{i=1}^\infty \|f_0- [b, \EuScript C](f_{j_i})\|_{L^p(bD,d\lambda)}^{p'}\bigg)^{1\over p'}\\
&\leq \|\{a_i\}\|_{\ell^p}\bigg(\sum_{i=1}^\infty 2^{-ip'}\bigg)^{1\over p'}\\
&\lesssim \|\{a_i\}\|_{\ell^p}.
\end{align*}
So this is a contradiction.
\fi

\smallskip

{\bf Necessary condition:}  
Recall that $\EuScript C=\EuScript C^\sharp+\EuScript R$. Since the kernel $C^\sharp(\cdot, \cdot)$ of $\EuScript C^\sharp$ is a standard kernel, by \cite[Theorem 1.1]{KL2}, $[b, \EuScript C^\sharp]$ is compact on 
$L^p(bD,d\lambda)$. Therefore, we only need to show that $[b, \EuScript R]$ is also compact on 
$L^p(bD,d\lambda)$.

From Lemma \ref {lemma 2 KL2}, for any $\xi>0$, there exists $b_\xi\in BUC(bD)$ such that 
$\|b-b_\xi\|_*<\xi$. Then
by Theorem 3.1 in \cite{KL1}, we have
\begin{align*}
\|[b, \EuScript R]f-[b_\xi, \EuScript R]f\|_{L^p(bD,d\lambda)}\leq C_p\|f\|_{L^p(bD,d\lambda)}\|b-b_\xi\|_*
<\xi C_p\|f\|_{L^p(bD,d\lambda)}
\end{align*}
Thus, to prove that $[b, \EuScript R]$ is compact on $L^p(bD,d\lambda)$, it suffices to prove that $[b_\xi, \EuScript R]$ is compact on $L^p(bD,d\lambda)$. 
By Lemma \ref {lemma 1 KL2} and \eqref {f_eta}, without loss of generality, we may assume that $b\in BUC(bD)$ and \eqref {f_eta 1} holds. 
By Lemma \ref {lemma 2 KL2}, it suffices to prove that for any fixed $\eta$ satisfying $0<\eta \ll 1$,  $[b, \EuScript R^\eta]$ is compact on $L^p(bD,d\lambda)$.

Since $R(w,z)$ is continuous on  $bD\times bD\backslash \{(z,z): z\in bD\}$, for any $f\in L^p(bD,d\lambda)$, we  see that $[b, \EuScript R^\eta]f$ is continuous on $bD$. 
\color{black}
To conclude the proof, we now argue that the image of the unit ball of $L^{p} (bD, d \lambda )$ under the commutator 
$[b, \EuScript R^\eta] $ is an equicontinuous family.  Compactness follows from the  Ascoli--Arzela theorem.  
\color{black} 

It remains to prove equicontinuity.  For any $z,w\in bD$ with ${\tt d}(w,z)<1$, we have
\begin{align*}
&[b, \EuScript R^\eta]f(z)-[b, \EuScript R^\eta]f(w)\\
&=b(z)\int_{bD}R^\eta(u,z)f(u)d\lambda(u)-\int_{bD}R^\eta(u,z)b(u)f(u)d\lambda(u)\\
&\quad -b(w)\int_{bD}R^\eta(u,w)f(u)d\lambda(u)+\int_{bD}R^\eta(u,w)b(u)f(u)d\lambda(u)\\
&=(b(z)-b(w))\int_{bD}R^\eta(u,z)f(u)d\lambda(u)+b(w)\int_{bD}\left(R^\eta(u,z)- R^\eta(u,w)\right)f(u)d\lambda(u)\\
&\quad+\int_{bD}\left(R^\eta(u,w)-R^\eta(u,z)\right)b(u)f(u)d\lambda(u)\\
&=(b(z)-b(w))\int_{bD}R^\eta(u,z)f(u)d\lambda(u)
+\int_{bD}\left(R^\eta(u,w)-R^\eta(u,z)\right)\left(b(u)-b(w)\right)f(u)d\lambda(u)\\
&=:I(z,w)+I\!I(z,w).
\end{align*}

For $I(z,w)$, by H\"older's inequality, we have
\begin{align*}
|I(z,w)|&=\left|(b(z)-b(w))\right| \left|\int_{bD}R^\eta(u,z)f(u)d\lambda(u)\right|\\
&\leq c\left|(b(z)-b(w))\right|\bigg(\int_{bD}\left|R^\eta(u,z)\right|^{p'}d\lambda(u) \bigg)^{1\over p'}\|f\|_{L^p(bD,d\lambda)}\\
&\leq {c \|f\|_{L^p(bD,d\lambda)}{\tt d}(w,z)^\epsilon},
\end{align*}
 {where the last inequality is due to the fact that $R^\eta(u,z)\in C(bD\times bD)$ and $bD$ is bounded.}

Since $b$ is bounded, if we let ${\tt d}(w,z)<{\eta\over c\cdot  c_R}$, by a discussion similar to \cite[p. 645]{KL2}, we can obtain that
\begin{align*}
|I\!I(z,w)|
&=\left| \int_{bD}\left(R^\eta(u,w)-R^\eta(u,z)\right)\left(b(u)-b(w)\right)f(u)d\lambda(u) \right|\\
 &\leq
c \|b\|_{L^\infty(bD,d\lambda)}\int_{bD\setminus B_{\eta\over c}(z)}{{\tt d}(w,z)\over {\tt d}(u,z)^{2n}} |f(u)|d\lambda(u)\\
&\leq 
c \|b\|_{L^\infty(bD,d\lambda)}\|f\|_{L^p(bD,d\lambda)}{\tt d}(w,z)
\left(\int_{bD\setminus B_{\eta\over c}(z)}{1\over {\tt d}(u,z)^{2np'}}d\lambda(u)\right)^{1\over p'}\\
&\leq
c \|b\|_{L^\infty(bD,d\lambda)}\|f\|_{L^p(bD,d\lambda)}{\tt d}(w,z)\big({\eta\over c} \big)^{-2n}\lambda(bD)^{1\over p'}\\
&\leq
c_{\eta,p}\|b\|_{L^\infty(bD,d\lambda)}\|f\|_{L^p(bD,d\lambda)}{\tt d}(w,z).
\end{align*}

Therefore, $\{[b, \EuScript R^\eta](\mathcal U)\}$ is an equicontinuous family, where $\mathcal U$ is the unit ball in $L^p(bD,d\lambda)$. 
This finishes the proof of (2) in Theorem \ref{cauchy}.
\end{proof}

{
We now provide an explanation of Remark \ref{rem1}.
\begin{proposition}\label{rem1prop}
Theorem \ref{cauchy} holds for the $L^p$, BMO and VMO spaces
with respect to any measure $d\omega$ on $bD$ of the form $d\omega = \omega d\sigma$, where 
the density $\omega$ is a strictly positive continuous function on $bD$.
\end{proposition}
\begin{proof}
We recall that the Leray--Levi measure $d\lambda$ played an important role in \cite{LS} since they needed to handle the adjoint of the Cauchy--Leray transform, which corresponds to $d\lambda$ but not the induced Lebesgue measure $d\sigma$ on $bD$. However, this distinction will not matter in the final statements of the $L^p(bD)$-boundedness of $\EuScript C$, since $d\lambda$ and $d\sigma$ are equivalent in the following sense (stated in Section 2.1 but we recall again here):
$$ d\lambda(w) =\Lambda(w)d\sigma(w), $$
where $\Lambda(w)$ is a continuous function on $bD$ such that $c\leq \Lambda(w)\leq \tilde c $ for all $w\in bD$ with $c,\tilde c$ two positive absolute constants. 

Now suppose $d\omega$ is a measure on $bD$ of the form $d\omega = \omega d\sigma$, where 
the density $\omega$ is a strictly positive continuous function on $bD$. Since $bD$ is bounded, we see that $c\leq \omega(w)\leq \tilde c $ for all $w\in bD$ with $c,\tilde c$ two positive absolute constants. Hence, $d\omega$ and $d\lambda$ are equivalent.  As a consequence, we see that 
the spaces $L^p(bD,d\omega)$, BMO$(bD,d\omega)$ and VMO$(bD,d\omega)$ are equivalent to 
$L^p(bD,d\lambda)$, BMO$(bD,d\lambda)$ and VMO$(bD,d\lambda)$.

Next, we recall that the proof of Theorem \ref{cauchy} depends only on the $L^p(bD)$ boundedness of the essential part $\EuScript C^\sharp$ and the remainder $\EuScript R$, as well as the pointwise  estimate of the kernels of $\EuScript C^\sharp$ and $\EuScript R$.  Hence, by repeating the proof of Theorem \ref{cauchy}, we see that Theorem \ref{cauchy} holds for the
$L^p(bD,d\omega)$, BMO$(bD,d\omega)$ and VMO$(bD,d\omega)$ spaces.
\end{proof}

We now provide an explanation to Remark \ref{rem1 family}.

\begin{proposition}\label{rem1prop epsilon}
Suppose $b\in {\rm BMO} (bD,d\lambda)$ and let $\{\EuScript C_\epsilon\}_\epsilon$ be a family of Cauchy integrals parameterized by $0<\epsilon< \epsilon_0$.  Then 
\begin{align}\label{upper bound epsilon}
\|[b,\EuScript C_\epsilon]\|_{L^p(bD, d\lambda)\rightarrow L^p(bD, d\lambda)} \leq C_\epsilon \|b\|_{{\rm BMO} (bD,d\lambda)},
\end{align}
where $C_\epsilon$ can go to $\infty$ as $\epsilon\to0$.
\end{proposition}
\begin{proof}
Recall that  in \cite{LS} it was pointed out that, for a family of Cauchy integrals $\{\EuScript C_\epsilon\}_\epsilon$ parameterized by $0<\epsilon< \epsilon_0$, then by splitting $\EuScript C_\epsilon =  \EuScript C^\sharp_\epsilon+ \EuScript R_\epsilon$, the operator norm $\|\EuScript R_\epsilon\|_{L^p(bD,d\lambda)\to L^p(bD,d\lambda)} $ can go to $\infty$ as $\epsilon\to0$.

To be more precise, we first recall that as pointed out in \cite{LS},
$$ c_\epsilon:= \sup_{\substack{w\in bD\\ 1\leq j,k\leq n} }| \nabla \tau_{j,k}^\epsilon(w) | $$
will in general tend to $\infty$ as $\epsilon\to0$, where $\tau_{j,k}^\epsilon(w)$ is the $n\times n$-matrix of $C^1$ function as stated in Subsection 2.1. Next, we recall the Cauchy--Fantappi\'e integral. Suppose $0<\epsilon< \epsilon_0$. Consider 
$$ G(w,z):= \chi(w,z) \ \Big[ \partial \rho(w) -{1\over2}\sum_{j,k} \tau_{j,k}^\epsilon(w) (w_j-z_j) dw_k 
\Big] + (1-\chi) \sum_k (\bar w_k- \bar z_k) dw_k, $$
where
$\chi(w,z)$ is a $C^\infty$-cutoff function
 \[
    \chi(w,z)=\left\{
                \begin{array}{ll}
                  1,\quad {\rm when}\ |w-z|\leq {\mu\over2};\\[6pt]
                  0,\quad {\rm when}\ |w-z|\geq {\mu}
                                  \end{array}
              \right.
  \] 
 (note that this 1-form $G$ was introduced in (3.1) in \cite{LS}). 
Let $g(w,z) = \langle G(w,z),w-z \rangle$. Then the Cauchy--Fantappi\'e integral ((3.2) in \cite{LS}) is defined as 
$$ \mathbf C^1(f)(z) :=  \int_{w\in bD} C^{(1)}(w,z) f(w),\quad z\in D,   $$
where
$$ C^{(1)}(w,z):= {1\over (2\pi i)^n} j^* \bigg( { G\wedge (\bar\partial G)^{n-1}(w,z) \over g(w,z)^n }  \bigg). $$
Another correction operator $\mathbf C^2$ is needed since $\mathbf C^1 (f)$ is not holomorphic for general $f$ (for details we refer to \cite[Subsection 3.2]{LS}). The kernel  $C^{(2)}(w,z)$ of $\mathbf C^2$ is bounded on $bD\times \bar D$ with 
$$\sup_{(w,z)\in bD\times\bar D} |C^{(2)}(w,z)|\lesssim1. $$
Then the Cauchy integral $\mathbf C = \mathbf C^1 +\mathbf C^2$.

Then the ``essential'' part is defined $\mathbf C^\sharp$ just by replacing ${1\over (2\pi i)^n} j^* \big(  G\wedge (\bar\partial G)^{n-1}(w,z)\big)$ by the Leray--Levi measure $d\lambda = {1\over (2\pi i)^n} j^* \big(  \partial \rho\wedge (\bar\partial \partial\rho)^{n-1}(w,z)\big)$, i.e.,
$$  \mathbf C^\sharp (f)(z) :=  \int_{w\in bD} { f(w)\over g(w,z)^n} f(w)d\lambda(w),\quad z\in D. $$
Hence, $\mathbf C$ can be written as 
$$\mathbf C = \mathbf C^\sharp +\mathbf R,$$
 where the remainder $\mathbf R$ is defined by $\mathbf R = \mathbf C^1 -\mathbf C^\sharp +\mathbf C^2$ with the explicit expression as 
$$\mathbf R(f)(z) =\int_{bD} { j^*\big( \alpha_0(w,z) + \sum_j\alpha_j(w,z) (w_j-z_j)  \big)  \over g(w,z)^n}  f(w)   +\mathbf C^2(f)(z),$$
where $\alpha_0,\alpha_j$ are $(2n-1)$-form with coefficient  on $w\in bD$ and smooth on $z\in \bar D$. Then for the fixed $\epsilon$ at the beginning, the remainder $ \EuScript R_\epsilon$ is defined as
$$ \EuScript R_\epsilon(f)(z) = \mathbf R(f)(z)\bigg|_{z\in bD}.$$
From the definition of $\mathbf R$ we see that restriction of $\mathbf C^2$ on $bD$ is uniformly bounded on $L^p(bD, d\lambda)$. However, for the first term, the upper bound of the coefficients $\alpha_0$ and $\alpha_j$ involves $ |\nabla \tau_{j,k}^\epsilon(w)|$. Hence, the 
operator norm $\|\EuScript R_\epsilon\|_{L^p(bD,d\lambda)\to L^p(bD,d\lambda)} $ is related to $c_\epsilon$.

 We now point out that, this fact will affect the $L^p(bD,d\lambda)\to L^p(bD,d\lambda)$ norm of the commutator $[b, \EuScript C_\epsilon]$ as well.  Recalling the proof of the necessity of $(1)$ in Theorem \ref{cauchy}, suppose $b\in $ BMO$(bD,d\lambda)$, then 
by writing $[b, \EuScript C_\epsilon] =[b, \EuScript C^\sharp_\epsilon] +[b, \EuScript R_\epsilon]$, 
it suffices to verify the $L^p(bD,d\lambda) \to L^p(bD,d\lambda)$ norms of 
$[b, \EuScript C^\sharp_\epsilon]$ and $[b, \EuScript R_\epsilon]$. Note that the first one is bounded by a constant multiplying $\|b\|_{{\rm BMO}(bD,d\lambda)}$ where the constant is independent of $\epsilon$. For the second one, by repeating the proof of 
\ref{[b,R] bounded} we see that
$$
\|[b,\EuScript R_\epsilon]\|_{L^p(bD, d\lambda)\rightarrow L^p(bD, d\lambda)} \leq C_\epsilon \|b\|_{{\rm BMO} (bD,d\lambda)},
$$
where $C_\epsilon$ is with respect to $c_\epsilon$, and hence it can go to $\infty$ as $\epsilon\to0$.
Hence, \eqref{upper bound epsilon} holds.
\end{proof}

}

\subsection{A remark on the Commutator $[b, \EuScript C^\sharp]$ }

We also point out that from the proof of the  main result above, we can also deduce that
the commutator of the essential part $\EuScript C^\sharp$ of $\EuScript C$ can also characterise the BMO space
on the boundary $bD$.

To be more precise, we have the following.
\begin{theorem}\label{csharp-iff}
Suppose $D\subset \mathbb C^n$, $n\geq 2$, is  a bounded domain whose boundary is of class $C^2$ and is strongly pseudoconvex,   $b\in L^1(bD, d\lambda)$. Then for $1<p<\infty$, 

{\rm (1)} the function $b\in{\rm BMO}(bD,d\lambda)$ if and only if the commutator $[b, \EuScript C^\sharp]$ is bounded on  $L^p(bD, d\lambda)$,

{\rm(2)} the function $b\in{\rm VMO}(bD,d\lambda)$ if and only if the commutator $[b, \EuScript C^\sharp]$ is compact on  $L^p(bD, d\lambda)$.
\end{theorem}

\begin{proof}

We point out that the proof of Theorem \ref{csharp-iff} follows from the proof of Theorem \ref {cauchy}, and in fact, it is simpler, since the operator $\EuScript C^\sharp$ is a Calder\'on--Zygmund operator. We only sketch the proof here.

Proof of (1):

It is clear that the necessary condition follows from the necessary part in the proof of Theorem \ref{cauchy} above.

The sufficient condition follows from the proof of the  sufficiency part of Theorem \ref{cauchy}  by mainly using the fact that at least one of the arguments $\mathfrak a)$ and
$\mathfrak b)$ holds (where $\mathfrak a)$ and
$\mathfrak b)$ are stated  in the proof of the  sufficiency part of Theorem \ref{cauchy} above)

Proof of (2): the proof follows from the proof of (2) of Theorem \ref{cauchy}.

This finishes the proof of Theorem \ref{csharp-iff}.
\end{proof}

In the end, we point out that similar to Remark \ref{rem1}, we see that Theorem \ref{csharp-iff} holds for the $L^p$, BMO and VMO spaces
with respect to any measure $d\omega$ on $bD$ of the form $d\omega = \omega d\sigma$, where 
the density $\omega$ is a strictly positive continuous function on $bD$.

\color{black}

\section{The Cauchy-Leray integral for domains in $\C^n$ with minimal smoothness}
\setcounter{equation}{0}

In this section, we focus on the bounded domain $D\subset \mathbb C^n$ which is  strongly $\mathbb C$-linearly convex and whose boundary satisfies the minimal regularity condition of class $C^{1,1}$.

\subsection{Preliminaries}
We now recall the notation from  \cite{LS2014}.
Suppose $D$ is a bounded domain in $\C^n$ with defining function $\rho$
satisfying 

1) $D$ is of class $C^{1,1}$, i.e., the first derivatives of its defining function $\rho$ are Lipschitz, and $|\nabla \rho(w)|>0$ whenever $w\in\{w:\rho(w)=0\}=bD$;

2) $D$ is strongly $\C$-linearly convex, i.e., $D$ is a bounded domain of $C^1$, and at any boundary point it satisfies either of the following two equivalent conditions
\begin{align*}
 |\Delta(w,z)|&\geq c|w-z|^2, \\%\quad {\rm if}\  \textcolor{red}{z\in D}, w\in bD,\\
 d_E\big(z,w+T_w^{\mathbb C}\big)&\geq \tilde c|w-z|^2, %\quad {\rm if}\   \textcolor{red}{z\in D}, w\in bD,
\end{align*}
for some $c,\tilde c>0$, where
\begin{align}\label{delta}
\Delta(w,z)=\langle\partial \rho(w), w-z\rangle,
\end{align}
and $d_E(z,w+T_w^{\mathbb C})$ denotes the Euclidean distance from $z$ to the affine subspace $w+T_w^{\mathbb C}$. Note that
$T_w^{\mathbb C}:=\{v:\langle\partial \rho(w),v\rangle=0\}$ is the complex tangent space referred to the origin, $w+T_w^{\mathbb C}$ is its geometric realization as an affine space tangent to $bD$ at $w$.

On $bD$ there is a quasi-distance ${\tt d}$, which is defined as 
$$ {\tt d}(w,z)= |\Delta(w,z)|^{1\over 2} = |\langle \partial\rho,w-z \rangle|^{1\over2},\quad w,z\in bD. $$

The Leray--Levi measure $d\lambda$ on $bD$ introduced in \cite{LS2014} is defined as that in Section 2. 
According to \cite[Proposition 3.4]{LS2014},  $d\lambda$ is also equivalent to the induced Lebesgue measure $d\sigma$ on $bD$ in the following sense:
$$ d\lambda(w) =\tilde\Lambda(w)d\sigma(w)\quad {\rm \ for\ }\sigma {\rm\ a.e.\ } w\in bD,$$
and there are two strictly positive constants $c_1$ and $c_2$ so that
$$ c_1\leq \tilde\Lambda(w)\leq c_2 \quad {\rm \ for\ }\sigma {\rm\ a.e.\ } w\in bD.$$

We also denote by $ B_r(w) =\{ z\in bD:\ {\tt d}(w,z)<r \}$  the  boundary balls determined via the quasidistance ${\tt d }$.
By \cite[Proposition 3.5]{LS2014}, we also have
\begin{align}\label{lambdab-1}
\lambda\big(B_r(w) \big)\approx r^{2n},\quad 0<r\leq 1.
\end{align}

The Cauchy--Leray integral of a suitable function $f$ on $bD$, denoted ${\bf C}(f)$, is formally defined by
$$  {\bf C}(f)(z) =\int_{bD} {f(w)\over \Delta(w,z)^n}d\lambda(w),\quad z\in D. $$

When restricting $z$ to the boundary $bD$, we have the Cauchy--Leray transform $f\mapsto \mathcal C(f)$, defined as
$$  {\mathcal C}(f)(z) =\int_{bD} {f(w)\over \Delta(w,z)^n}d\lambda(w),\quad z\in bD, $$
where the function $f$ satisfies the H\"older-like condition
$$ |f(w_1)-f(w_2)|\lesssim {\tt d}(w_1,w_2)^\alpha,\quad %{\rm for\ all\ }
 w_1,w_2\in bD, $$
  for some $0<\alpha\leq 1$.

The main result in Lanzani--Stein \cite{LS2014} is as follows:

\begin{theorem}[Lanzani--Stein \cite{LS2014}, Theorem 5.1]\label{thm2.1}
The Cauchy--Leray transform $f\mapsto \mathcal C(f)$, initially defined for functions $f$ that satisfy the H\"older-like condition for some $\alpha$, extends to a bounded linear operator on $L^p(bD)$ for $1<p<\infty$.
\end{theorem}

To obtain the $L^p(bD)$  boundedness, the main approach that Lanzani--Stein \cite{LS2014} used is first to obtain 
the kernel estimate for the Cauchy--Leray transform $ \mathcal C$ and then to use the $T(1)$ theorem. To be more specific,
let us take $K(w,z)$ to be the function defined for $w,z\in bD$, with $w\not= z$, by
$$K(w,z) = {1\over \Delta(w,z)^n}.$$
This function is the ``kernel" of the operator  $\mathcal C$, in the sense that
$$  \mathcal C(f)(z)= \int_{bD} K(w,z)f(w)d\lambda(w), $$
whenever $z$ lies outside of the support of $f$ and $f$ satisfies the H\"older-like condition for some $\alpha$.
The size and regularity estimates that are relevant for us are:
\begin{align}\label{ck}
& |K(w,z)|\lesssim {1\over{\tt d}(w,z)^{2n}};\nonumber\\
& |K(w,z) - K(w',z)|\lesssim  { {\tt d}(w,w')\over  {\tt d}(w,z)^{2n+1} },\quad {\rm if}\ {\tt d}(w,z)\geq c_K{\tt d}(w,w');\\
& |K(w,z) - K(w,z')|\lesssim  { {\tt d}(z,z')\over  {\tt d}(w,z)^{2n+1} },\quad {\rm if}\ {\tt d}(w,z)\geq c_K {\tt d}(z,z'),\nonumber
\end{align}
for an appropriate constant $c_K>0$. Moreover, for the size estimates we actually have
\begin{align}\label{kd}
 |K(w,z)|= {1\over{\tt d}(w,z)^{2n}}.
 \end{align}

\subsection{Boundedness and compactness of the Commutator $[b,\mathcal C]$}

\begin{proof}[Proof of Theorem \ref {Cauchy-Leray}]

We point out that the proof of Theorem \ref {Cauchy-Leray} follows from the proof of Theorem \ref {cauchy}, and in fact, it is simpler, since the operator $\mathcal C$ is a Calder\'on--Zygmund operator. We only sketch the proof here.

We first prove (1) of Theorem \ref {Cauchy-Leray}.

From Theorem \ref{thm2.1} and the size and smoothness conditions of the kernel $K$ above, we see that $\mathcal C$ is a Calder\'on--Zygmund operator. According to {\cite[Theorem 3.1]{KL1}},  
if $b\in{\rm BMO}(bD,d\lambda)$, we can obtain that $[b, \mathcal C]$ is bounded on  $L^p(bD, d\lambda)$.
Thus, it suffices to verify the sufficient condition.

To see this, assume that $b$ is in $L^1(bD,d\lambda)$ and that
$\big\| [b,\mathcal C] \big\|_{L^p(bD, d\lambda)\rightarrow L^p(bD, d\lambda)}<\infty$.
We now verify that $b$ is in ${\rm BMO}(bD,d\lambda)$.

In fact, by using the size and smoothness estimates as in \eqref{kd} and \eqref{ck}, following the argument in the proof of the sufficiency part of Theorem \ref{cauchy}, we see that there exist positive constants $\gamma_1$, $A_4$ and $A_5$ such that for every ball $B=B_r(z_0)\subset bD$ with $r<\gamma_1$,
there is another ball $\tilde B=B_r(w_0) \subset bD$ with 
$A_4r\leq{\tt d}(w_0,z_0)\leq (A_4+1) r$ such that at least one of the following properties holds:

$\mathfrak c)$ for every $z\in B$ and $w\in \tilde B$, $ K_1(w,z)$ does not change sign and $$| K_1(w,z)|\geq A_5 { 1\over  {\tt d}(w,z)^{2n} };$$

$\mathfrak d)$ for every $z\in B$ and $w\in \tilde B$, $ K_2(w,z)$ does not change sign and $$| K_2(w,z)|\geq A_5  {1\over  {\tt d}(w,z)^{2n} },$$
where $K_1(w,z)$ and $K_2(w,z)$ are the real and imaginary parts of $K(w,z)$, respectively. Here the constants $\gamma_1,$ $A_4$ and $A_5$ depend only on the implicit constants in \eqref{ck} and the quasi-distance ${\tt d}$.

  We test the ${\rm BMO}(bD,d\lambda)$ condition on the case of balls with big radius and small radius.

Case 1: In this case we work with balls with a large radius, $r\geq  \gamma_1$. 

By \eqref{lambdab} and by the fact that 
$\lambda(B)\geq \lambda( B_{ \gamma_1}(z_0)) \approx  \gamma_1^{2n}$, we obtain that
\begin{align*}%\label{bmo norm1}
{1\over \lambda(B)} \int_B|b(w)-b_B|d\lambda(w)\lesssim  \gamma_1^{-2n} \|b\|_{L^1(bD, d\lambda)}.
\end{align*}

Case 2: In this case we work with balls $B$ with a small radius, $r< \gamma_1$.

Then, by using one of the arguments $\mathfrak c)$ and $\mathfrak d)$,  following the approach of the proof of Case 2 in the proof of Theorem \ref{cauchy}, we obtain that 
\begin{align*}
{1\over \lambda(B)} \int_B|b(w)-b_B|d\lambda(w)\lesssim \left\|[b,\mathcal C]\right\|_{L^p(bD, d\lambda)\to L^p(bD, d\lambda)}.
\end{align*}
This finishes the proof of (1) of Theorem \ref {Cauchy-Leray}.

Proof of (2): Necessary condition:  Since the kernel $K(\cdot, \cdot)$ of $\mathcal C$ is a standard kernel, by \cite[Theorem 1.1]{KL2}, $[b, \mathcal C]$ is compact on 
$L^p(bD,d\lambda)$. 
 The sufficient condition follows from the proof of (2) of Theorem \ref{Cauchy-Leray}.
\end{proof}

{We now provide a strict explanation to Remark \ref{rem2}.
\begin{proposition}\label{rem2prop}
Theorem \ref{Cauchy-Leray} holds for the $L^p$, BMO and VMO spaces
with respect to any measure $d\omega$ on $bD$ of the form $d\omega = \omega d\sigma$, where 
the density $\omega$ is a strictly positive continuous function on $bD$.
\end{proposition}
\begin{proof}
The proof of this proposition is similar to that of Proposition \ref{rem2prop}.
Suppose $d\omega$ is a measure on $bD$ of the form $d\omega = \omega d\sigma$, where 
the density $\omega$ is a strictly positive continuous function on $bD$. 
Following the proof of Proposition \ref{rem2prop}, we see that
the spaces 
$L^p(bD,d\omega)$, BMO$(bD,d\omega)$ and VMO$(bD,d\omega)$ are equivalent to 
$L^p(bD,d\lambda)$, BMO$(bD,d\lambda)$ and VMO$(bD,d\lambda)$.

Next, we recall that the proof of Theorem \ref{Cauchy-Leray} depends only on the $L^p(bD)$ boundedness of the Cauchy--Leray transform $\EuScript C$ and the pointwise  estimate of its kernel.

Hence, by repeating the proof of Theorem \ref{Cauchy-Leray}, we see that it holds for the
$L^p(bD,d\omega)$, BMO$(bD,d\omega)$ and VMO$(bD,d\omega)$ spaces.
\end{proof}
}

\section{A remark on the commutator of Cauchy--Szeg\H o operator on a bounded strictly pseudoconvex domain in $\mathbb C^n$ with smooth boundary}
\setcounter{equation}{0}

Let $D$ be a strictly pseudoconvex domain in $\C^n $ with smooth boundary and $\rho(z)$ be   
 a strictly pluri-superharmonic defining function for $D$.
Set
$$\psi(z,w)=\sum_{j=1}^n{\partial\rho\over \partial w_j}(z_j-w_j) +{1\over2} \sum_{j,k} {\partial^2\rho(w) \over \partial w_j \partial w_k} (z_j-w_j)(z_k-w_k), $$
 Then there is a positive number $\tilde\delta>0$ such that the Cauchy--Szeg\H o kernel $S(z,w)$ has the following form
 \begin{align}\label{szw}
 S(z,w)=F(z,w)\psi(z,w)^{-n}+G(z,w)\log\psi(z,w)
 \end{align}
 for all $(z,w)\in R_{\tilde \delta}=\{(z,w)\in bD\times bD: d(z,w)<\tilde\delta\}$, where $F, G\in C^\infty(bD\times bD)$ and $F(z,z)>0$ on $bD$, $d$ is the usual quasi-distance on $bD$ {(see for example \cite[p. 33]{STE2}).}
 According to \cite{FEF, NRSW, BMS}, the Cauchy--Szeg\H o kernel $S(z,w)\in C^\infty(bD\times bD\setminus \{(z,z): z\in bD\})$ is a standard kernel.

We define the Cauchy--Szeg\H o operator $T_S$ as the singular integral associated with the Cauchy--Szeg\H o kernel $S(z,w)$, i.e. 
$$T_S(f)(z)=\int_{bD}S(z,w)f(w)d\mu(w),$$
for suitable $f$ on $bD$, where $d\mu$ is the usual Lebesgue--Hausdorff surface measure on $bD$. We still use  $B_r(z)$ to denote the ball on $bD$ determined by the quasi-distance $d$, then $\mu(B_r(z))\approx r^n$ (c.f. \cite[p. 34]{STE2}, {\cite {KL2}}).

 From \eqref{g0}, we can see that if $|w-z|\leq \mu_0/2$ for some fixed small $\mu_0>0$, then 
$|\psi(z,w)|=|g_0(w,z)|$. 
 
We now provide another proof of Theorem A as stated in the introduction. 

\begin{proof}[Proof of Theorem A]
 
We point out that the proof here follows from the proof of Theorem \ref {cauchy}, and in fact, it is simpler, since the operator $S$ is a Calder\'on--Zygmund operator with a specific kernel. We only sketch the proof here.

%\color{red}

Proof of (i):

It suffices to show the sufficient condition. To see this, assume that $b$ is in $L^1(bD)$ and that
$\big\| [b, T_S] \big\|_{L^p(bD)\rightarrow L^p(bD)}<\infty$.

Since $S(z,w)$ is a standard kernel, there exist positive constants $\epsilon$ and $A_6$ such that for every $(z,w)\in R_{\tilde \delta}$ with $w\not=z$,
\begin{align}\label{swz kernel}
\left\{
                \begin{array}{ll}
                  a)'\ \ |S(z,w)|\leq A_6 {\displaystyle1\over d(w,z)^{n}};\\[5pt]
                  b)'\ \ |S(z,w) - S(z,w')|\leq A_6 \big({ d(w,w')\over  d(w,z)}\big)^{\epsilon}{ 1\over  d(w,z)^{n} },\quad {\rm if}\ d(w,z)\geq c{ d}(w,w');\\[5pt]
                  c)'\ \ |S(z,w) - S(z',w)|\leq A_6  \big({ d(z,z')\over  d(w,z)}\big)^{\epsilon}{ 1\over  d(w,z)^{n} },\quad {\rm if}\ { d}(w,z)\geq c { d}(z,z')
                \end{array}
              \right.
\end{align}
for an appropriate constant $c>0$.

Then similar to the proof of Theorem \ref{cauchy}, there exist positive constants $\gamma_2$, $A_7$ and $A_8$ such that 
for every ball $B=B_r(z_0)\subset bD$ with $r<\gamma_2$, there exists another ball
$\tilde B = B_r(w_0)\subset bD$ with  $A_7r \leq {\tt d}(w_0,z_0) \leq (A_7+1)r$,
such that
at least one of the following properties holds:

$\mathfrak e)$ for every $z\in B$ and $w\in \tilde B$, $ S_1(z,w)$ does not change sign and $$| S_1(z,w)|\geq A_8{ 1\over  { d}(z,w)^{n} };$$

$\mathfrak f)$ for every $z\in B$ and $w\in \tilde B$, $ S_2(z,w)$ does not change sign and $$| S_2(z,w)|\geq A_8{ 1\over {d}(z,w)^{n} },$$
 where $S_1(z,w)$ and $S_2(z,w)$ are the real and imaginary parts of $S(z,w)$, respectively.

We test the ${\rm BMO}(bD,d\lambda)$ condition on the case of balls with big radius and small radius.

Case 1: In this case we work with balls with a large radius, $r\geq  \gamma_2$. 

By the fact that 
$\mu(B)\geq \mu( B_{  \gamma_2}(z_0)) \approx   \gamma_2^{n}$, we obtain that
\begin{align*}%\label{bmo norm1}
{1\over \mu(B)} \int_B|b(w)-b_B|d\mu(w)\lesssim   \gamma_2^{-n} \|b\|_{L^1(bD)}.
\end{align*}

Case 2: In this case we work with balls $B$ with a small radius, $r<  \gamma_2$.

Then, following the approach of the proof of Case 2 in the proof of Theorem \ref{cauchy}, we obtain that 
\begin{align*}
{1\over \mu(B)} \int_B|b(w)-b_B|d\mu(w)\lesssim \left\|[b,T_S]\right\|_{L^p(bD)\to L^p(bD)}.
\end{align*}

Proof of (ii): the proof follows from the approach in the proof of (2) of Theorem \ref{cauchy} together with $\mathfrak e)$ or $\mathfrak f)$.

\color{black}

This finishes the proof of Theorem A.
\end{proof}

\bigskip

{\bf Acknowledgement:}
The authors would like to thank the referee for the helpful comments and suggestions on math and English, which makes the paper more accurate and clear.

 X. T. Duong, M. Lacey and J. Li are supported by ARC DP 160100153. X. Duong is also supported by Macquarie University Internal Grant 82184614.  J. Li is also supported by
Macquarie University Research Seeding Grant. 
 B. D. Wick's research is partially supported by National Science Foundation -- DMS \# 1800057 and DMS \# 1560955.
 Q. Y. Wu is supported by NSF of China (Grants No. 11671185 and No. 11701250), the Natural Science Foundation of Shandong Province (No. ZR2018LA002 and No. ZR2019YQ04) and the State Scholarship Fund of China (No. 201708370017).

\bibliographystyle{amsplain}

\end{document}